\documentclass[journal]{IEEEtran}

\usepackage{enumerate}
\usepackage[OT1]{fontenc}
\usepackage[usenames]{color}
\usepackage[colorlinks,
linkcolor=red,
anchorcolor=blue,
citecolor=blue
]{hyperref}
\usepackage{mathrsfs}
\usepackage{amsfonts,amsmath,amssymb,amsthm,url,xspace}
\usepackage{hyperref}
\usepackage{float}
\usepackage{tikz}
\usetikzlibrary{arrows.meta}
\usetikzlibrary{decorations.pathreplacing}
\usepackage{mathtools}
\usepackage{enumitem} 
\usepackage{graphicx}
\usepackage{savesym}
\savesymbol{AND}
\usepackage{algorithmic,algorithm}
\usepackage{cite}
\usepackage[protrusion=true, expansion=true]{microtype}
\usepackage{subfigure}
\usepackage{verbatim}
\usetikzlibrary{arrows,shapes}
\usepackage{mathtools}
\usepackage{authblk}
\usepackage[bottom]{footmisc}
\usepackage{arydshln,leftidx,mathtools}
\usepackage{lscape}

\newtheorem{definition}{Definition}
\newtheorem{assum}{Assumption}
\newtheorem{thm}{Theorem}

\newtheorem{lem}{Lemma}
\newtheorem{prop}{Proposition}

\newcommand{\st}{\mathop{\textrm{\normalfont s.t.}}}

\newcommand{\red}{\color{black}}

\newcommand{\mR}{\mathbb{R}}
\newcommand{\mZ}{\mathbb{Z}}

\newcommand{\mP}{\mathcal{P}}

\newcommand{\mC}{\mathcal{C}}
\newcommand{\mL}{\mathcal{L}}

\newcommand{\mQ}{\mathcal{Q}}

\def\P{{\mathcal P}}

\def\0{{\boldsymbol 0}}
\def\ba{{\boldsymbol{a}}}

\def\N{{\mathcal N}}

\def\bp{{\boldsymbol{p}}}
\def\bq{{\boldsymbol{q}}}
\def\bv{{\boldsymbol{v}}}
\def\bx{{\boldsymbol{x}}}
\def\br{{\boldsymbol{r}}}
\def\bs{{\boldsymbol{s}}}
\def\bv{{\boldsymbol{v}}}

\def\bu{{\boldsymbol{u}}}
\def\bw{{\boldsymbol{w}}}
\def\bd{{\boldsymbol{d}}}

\def\bl{{\boldsymbol{l}}}
\def\be{{\boldsymbol{e}}}
\def\bz{{\boldsymbol{z}}}

\def\blambda{{\boldsymbol{\lambda}}}
\def\bzeta{{\boldsymbol{\zeta}}}
\def\bxi{{\boldsymbol{\xi}}}
\def\tlambda{{\boldsymbol{\lambda}}^\star}

\def\tx{{\boldsymbol{x}}^\star}
\def\tu{{\boldsymbol{u}}^\star}
\def\tz{{\boldsymbol{z}}^\star}
\def\tw{{\boldsymbol{w}}^\star}
\def\bblambda{\bar{\boldsymbol{\lambda}}}

\def\bbx{\bar{\boldsymbol{x}}}
\def\bbu{\bar{\boldsymbol{u}}}

\def\bbw{\bar{\boldsymbol{w}}}

\def\tp{{\boldsymbol{p}}^\star}
\def\tq{{\boldsymbol{q}}^\star}
\def\tzeta{{\boldsymbol{\zeta}}^\star}
\def\txi{{\boldsymbol{\xi}}^\star}
\def\tzetac{{\boldsymbol{\zeta}}^{c \star}}
\def\txic{{\boldsymbol{\xi}}^{c \star}}
\def\tpc{{\boldsymbol{p}}^{c \star}}
\def\tqc{{\boldsymbol{q}}^{c \star}}

\newcommand{\rbr}[1]{\left(#1\right)}

\newcommand{\cbr}[1]{\left\{#1\right\}}

\allowdisplaybreaks

\def\diag{{\text{diag}}}

\begin{document}
\title{On the Convergence of Overlapping Schwarz \\ Decomposition for Nonlinear Optimal Control}

\author{Sen Na, Sungho Shin, Mihai Anitescu, and Victor M. Zavala
\thanks{S. Na and S. Shin equally contributed to this work.}
\thanks{S. Na is with the Department of Statistics, University of Chicago, 5747 South Ellis Avenue, Chicago, IL 60637, USA (e-mail: senna@uchicago.edu)}
\thanks{S. Shin is with the Department of Chemical and Biological Engineering, University of Wisconsin-Madison, Madison, WI 53706 USA (e-mail: sungho.shin@wisc.edu)}
\thanks{M. Anitescu is with the Mathematics and Computer Science Division, Argonne National Laboratory, Lemont, IL 60439, USA, and also with the Department of Statistics, University of Chicago, Chicago, IL 60637, USA (e-mail: anitescu@mcs.anl.gov)}
\thanks{V. M. Zavala is with the Department of Chemical and Biological Engineering, University of Wisconsin-Madison, Madison, WI 53706 USA and also with the Mathematics and Computer Science Division, Argonne National Laboratory, Lemont, IL 60439, USA (e-mail: victor.zavala@wisc.edu)}}

\maketitle

\begin{abstract}

We study the convergence properties of an overlapping Schwarz decomposition algorithm for solving nonlinear optimal control problems (OCPs). The algorithm decomposes the~time domain into a set of overlapping subdomains, and~solves all~subproblems defined over subdomains in parallel. The~convergence is attained by updating primal-dual information~at~the boundaries of overlapping subdomains. We show that the algorithm exhibits local linear convergence, and that the convergence rate improves exponentially with the overlap size. We also establish global convergence results for a general quadratic programming, which enables the application of the Schwarz scheme inside~second-order optimization algorithms (e.g., sequential quadratic programming). {\red The theoretical foundation~of our convergence analysis} is a~sensitivity result of {\red nonlinear} OCPs, which we call ``exponential decay of sensitivity" (EDS). Intuitively, EDS states that the impact of perturbations at domain~boundaries (i.e. initial and terminal time) {\red on the solution} decays exponentially as one moves into~the domain. {\red Here, we expand a previous analysis available in the literature~by showing} that EDS holds for {\red both primal and dual solutions}~of nonlinear OCPs, under uniform second-order sufficient condition, controllability condition, and boundedness condition. We conduct experiments with a quadrotor motion planning problem and a PDE control problem {\red to validate our theory}; and show that~the approach is significantly more efficient than ADMM and as efficient as the centralized solver Ipopt. 
\end{abstract}

\IEEEpeerreviewmaketitle

\section{Introduction}\label{sec:1}

We study the nonlinear optimal control problem (OCP):
\begin{subequations}\label{pro:1}
\begin{align}
\min_{\{\bx_k\},\{\bu_k\}}\  & \sum_{k=0}^{N-1} g_k(\bx_k, \bu_k) + g_N(\bx_N), \label{pro:1a}\\
\st\;\; & \bx_{k+1} = f_k(\bx_k,\bu_k)  \hskip0.43cm (\blambda_k), \label{pro:1b}\\
& \bx_0 = \bar{\bx}_0  \hskip1.8cm (\blambda_{-1}), \label{pro:1c}
\end{align}
\end{subequations}
where $\bx_k\in\mR^{n_x}$ are the state variables; $\bu_k\in\mR^{n_u}$ are the control variables; $\blambda_k\in\mR^{n_x}$ are the dual variables associated with the dynamics \eqref{pro:1b}; $\blambda_{-1}\in\mR^{n_x}$ are the dual variables associated with the initial conditions \eqref{pro:1c}; $g_k: \mR^{n_x}\times \mR^{n_u}\rightarrow \mR$ ($g_N:\mR^{n_x}\rightarrow\mR$) are the cost functions; $f_k:\mR^{n_x}\times \mR^{n_u}\rightarrow \mR^{n_x}$ are the dynamical constraint functions; $N$ is the horizon length; and $\bbx_0\in\mR^{n_x}$ is the given initial state. We assume that $f_k$, $g_k$ are twice continuously differentiable, nonlinear, and~possibly nonconvex; as such, \eqref{pro:1} is a nonconvex nonlinear program (NLP). The problem of interest has been studied extensively in the context of model predictive control \cite{Rawlings2017Model} with applications in chemical process control \cite{Qin2003survey}, energy systems \cite{Kumar2020Stochastic}, production planning\cite{Jackson2003Temporal}, autonomous vehicles \cite{Falcone2007Predictive}, power systems \cite{Shanechi2003General}, supply chains \cite{Dunbar2007Distributed}, and neural networks \cite{Huang2003Neural}.

In this work, we are interested in solving OCPs with a large number of stages $N$. Such problems arise in the settings with long horizons, fine time discretization resolutions, and multiple timescales \cite{Kumar2018Handling}. Temporal decomposition provides an approach to deal with such problems. In this approach, one partitions~the time domain $[0,N]$ into a set of subdomains $\{[m_i,m_{i+1}]\}_{i=0}^{T-1}$. One then solves {\red more tractable subproblems} over subdomains in parallel, and their solution trajectories are concatenated~by using a coordination mechanism. Traditional coordination~mechanisms include Lagrangian dual decomposition \cite{Beccuti2004Temporal}, alternating direction method of multipliers (ADMM) \cite{Boyd2010Distributed}, dual dynamic programming \cite{Kumar2018Stochastic}, and Jacobi/Gauss-Seidel~methods~\cite{Zavala2016New}. These decomposition approaches offer flexibility in that they can be implemented in different types of computing hardware that might have limitations on memory and processor speeds. This is critical because the performance of centralized nonlinear optimization solvers (e.g., Ipopt) degrades rapidly in resource-constrained computing environments \cite{Chiang2017Augmented}. Unfortunately, while Lagrangian dual decomposition, ADMM, and dual dynamic programming are guaranteed to converge under a variety of OCP settings, they often exhibit slow convergence \cite{Kozma2014Benchmarking}. This highlights {\red the existence of} a fundamental trade-off between the flexibility offered by distributed solvers and the efficiency offered by centralized solvers.
    
Direct decomposition approaches have also been studied for convex OCPs with long horizons. Specifically, such approaches have been used to decompose linear algebra systems inside interior-point solvers \cite{Nielsen2014O, Nielsen2015parallel, Laine2019Parallelizing, Wright1990Solution, Rao1998Application, Wan2019Parallel, Kang2015Nonlinear, Frison2013Efficient}. They also offer flexibility to enable limited-resource-hardware implementations and, since the methods are direct (as opposed to iterative), they do not suffer from convergence issues. However, direct approaches rely on  reduction procedures (they are block elimination techniques), and such procedures suffer from scalability issues. For instance, parallel cyclic reduction, Schur, and Riccati decompositions do not scale well with the number of states and/or control variables.  Moreover, we also highlight that iterative approaches such as ADMM and Lagrangian dual decomposition often offer more flexibility than direct decomposition methods in that the amount of communication needed is limited (thus preserving data privacy). 
%Furthermore, such approaches are more amenable to asynchronous computations. 

A recent study \cite{Barrows2014Time} has empirically tested the effectiveness of a different decomposition paradigm for OCPs. Specifically, the authors performed numerical tests with a {\it temporal decomposition scheme with overlaps} (see Fig. \ref{fig:schematic}).  Here, overlapping subdomains $\{[n_i^1,n_i^2]\}_{i=0}^{T-1}$ are constructed by expanding the non-overlapping subdomains $\{[m_i,m_{i+1}]\}_{i=0}^{T-1}$ by $\tau$ stages on the left and right boundaries. Subproblems on the expanded subdomains are solved {\red in parallel}, and the resulting solution trajectories are concatenated by discarding the pieces of the trajectory in the overlapping regions. The authors observed~that, as \textit{the size of the overlap} increases, the approximation error of the concatenated solution trajectory drops rapidly. {\red However, no quantitative analysis was provided. Subsequent work \cite{Xu2018Exponentially} provided the first rigorous error analysis} of such overlapping decomposition scheme. The authors~proved~that,~for~{\red strongly~convex OCPs  with linear dynamics and positive-definite quadratic stage costs} that satisfy uniform controllability and boundedness conditions, the error of the concatenated trajectory decreases {\em exponentially} in $\tau$. This result requires a sensitivity property for convex OCPs that we call ``exponential decay of sensitivity" (EDS). This property says that the impact of parametric perturbations on the primal {\red solution} trajectory $\{(\tx_k,\tu_k)\}_k$ decays exponentially as one moves away from the perturbation stage. Unfortunately,~{\red the analysis in \cite{Xu2018Exponentially} does not apply for~the general nonlinear OCP \eqref{pro:1} and thus has limited applicability. 
Furthermore, we emphasize that the sensitivity on the dual solution (even for convex case) is not resolved in that work, and that the decomposition scheme analyzed there is only an approximation scheme (not a convergent algorithm).  

}

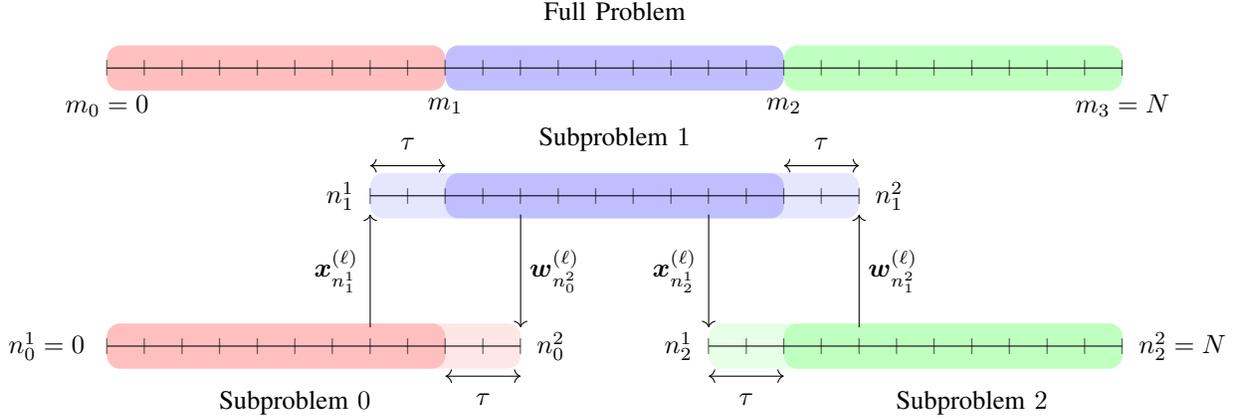
\begin{figure*}[t]
\centering
\begin{tikzpicture}
% full problem
\def\shift{.8};
\fill[rounded corners=5pt,red!25!white,line width=3pt](0,.2-\shift) rectangle ++(4.5,.6);
\fill[rounded corners=5pt,blue!25!white,line width=3pt](4.5,.2-\shift) rectangle ++(4.5,.6);
\fill[rounded corners=5pt,green!25!white,line width=3pt](9,.2-\shift) rectangle ++(4.5,.6);
\foreach \x in {1,...,28}
\node[scale=.6]  at (\x/2-1/2,.5-\shift) {$|$};
\draw (0,.5-\shift)--(13.5,.5-\shift);
\node at (0,0-\shift) {$m_0=0$};
\node at (4.5,0-\shift) {$m_1$};
\node at (9,0-\shift) {$m_2$};
\node at (13.5,0-\shift) {$m_3=N$};
\node at (6.75,1.25-\shift) {Full Problem};

% subproblem 0
\def\spo{.8};
\def\spt{.4};
\fill[rounded corners=5pt,red!10!white,line width=3pt](0,-4.3) rectangle ++(5.5,.6);
\fill[rounded corners=5pt,red!25!white,line width=3pt](0,-4.3) rectangle ++(4.5,.6);
\foreach \x in {1,...,12}
\node[scale=.6]  at (\x/2-1/2,-4) {$|$};
\draw (0,-4)--(5.5,-4);
\node at (0-\spo,-4) {$n^1_0=0$};
\node at (5.5+\spt,-4) {$n^2_0$};
\node at (2.5,-4.75) {Subproblem $0$};
\draw[<->] (4.5,-4.4)--(5.5,-4.4);\node at (5,-4.7) {$\tau$};

% subproblem 1
\fill[rounded corners=5pt,blue!10!white,line width=3pt](3.5,-2.3) rectangle ++(6.5,.6);
\fill[rounded corners=5pt,blue!25!white,line width=3pt](4.5,-2.3) rectangle ++(4.5,.6);
\foreach \x in {1,...,14}
\node[scale=.6]  at (\x/2+6/2,-2) {$|$};
\draw (3.5,-2)--(10,-2);
\node at (3.5-\spt,-2) {$n_1^1$};
\node at (10+\spt,-2) {$n^2_1$};
\node at (6.75,-1.25) {Subproblem $1$};
\draw[<->] (3.5,-1.6)--(4.5,-1.6);\node at (4,-1.3) {$\tau$};
\draw[<->] (9,-1.6)--(10,-1.6);\node at (9.5,-1.3) {$\tau$};

% subproblem 2
\fill[rounded corners=5pt,green!10!white,line width=3pt](8,-4.3) rectangle ++(5.5,.6);
\fill[rounded corners=5pt,green!25!white,line width=3pt](9,-4.3) rectangle ++(4.5,.6);
\foreach \x in {1,...,12}
\node[scale=.6]  at (\x/2+15/2,-4) {$|$};
\draw (8,-4)--(13.5,-4);
\node at (8-\spt,-4) {$n_2^1$};
\node at (13.5+\spo,-4) {$n_2^2=N$};
\node at (11.5,-4.75) {Subproblem $2$};
\draw[<->] (8,-4.4)--(9,-4.4);\node at (8.5,-4.7) {$\tau$};

% arrows
\def\mg{0.25};
\def\sh{0.45};
\draw[<-] (5.5,-4+\mg)--(5.5,-2-\mg);
\node at (5.5+\sh,-3) {$\bw^{(\ell)}_{n_0^2}$};
\draw[->] (3.5,-4+\mg)--(3.5,-2-\mg);
\node at (3.5-\sh,-3) {$\bx^{(\ell)}_{n^1_1}$};
\draw[->] (10,-4+\mg)--(10,-2-\mg);
\node at (10+\sh,-3) {$\bw^{(\ell)}_{n_1^2}$};
\draw[<-] (8,-4+\mg)--(8,-2-\mg);
\node at (8-\sh,-3) {$\bx^{(\ell)}_{n_2^1}$};
\end{tikzpicture}
\caption{Overlapping Schwarz decomposition scheme for OCPs. Here, $m_i$ denotes non-overlapping subdomains boundaries; $n_i^1, n_i^2$ denote left and right overlapping subdomains boundaries; $\bx_k^{(\ell)}$ is the state iterate at stage $k$ in the $\ell$-th iteration; $\bw_k^{(\ell)} = (\bx_k^{(\ell)}; \bu_k^{(\ell)}; \blambda_k^{(\ell)})$ is the primal-dual iterate at stage $k$ in the $\ell$-th iteration. Each subproblem depends on the initial state iterate coming from the previous subproblem, and the terminal primal-dual iterate coming from the next subproblem.}\label{fig:schematic}
\end{figure*}

Recent work {\red\cite{Shin2019Parallel} has applied the overlapping decomposition scheme for solving time-invariant nonlinear OCPs}. This relies on the observation that such a decomposition scheme can be interpreted as a {\em single~iteration} of an overlapping Schwarz decomposition scheme. {\red In particular, for solving nonlinear OCPs,} \cite{Shin2019Parallel} partitions the time domain~as in \cite{Barrows2014Time, Xu2018Exponentially}, but utilizes~both~primal~and~dual information from adjacent subdomains to perform {\red an iterative} coordination to achieve the convergence. 
%This algorithm aims to solve the full-horizon problem exactly by contracting the error to zero, whereas two previous works \cite{Barrows2014Time, Xu2018Exponentially} only approximately solve the full-horizon problem. 
The authors of \cite{Shin2019Parallel} {\red conjectured that the effect of perturbations at two ends (that is initial and terminal stages) on the {\em primal and dual} trajectory $\{(\tx_k,\tu_k,\tlambda_k)\}_k$ (not only for the primal trajectory as in~\cite{Xu2018Exponentially}) decays \textit{asymptotically}. Under this conjecture, they proved~that the overlapping Schwarz scheme converges locally.} The authors also provided empirical evidence with a nonlinear OCP that, {\red the perturbation effect decays not only asymptotically, but indeed exponentially. That is, EDS empirically holds for nonlinear OCPs just like for convex quadratic~OCPs as in \cite{Xu2018Exponentially}, although} a theoretical justification for such behavior {\red was not provided.} 

{\red The work in \cite{Na2020Exponential} investigated primal sensitivity for nonlinear OCPs. The authors showed that}, under uniform second-order sufficient condition, controllability condition, and boundedness condition, {\red EDS holds for primal solution of nonlinear OCPs. This result generalizes the convex setup~in \cite{Xu2018Exponentially} to a general nonconvex nonlinear setup under the same conditions. The generalization relies on~a~convexification~technique, which convexifies nonconvex problems to convex problems without altering primal solutions (cf. Algorithm \ref{alg:convex:proce}). However, the result in \cite{Na2020Exponential} is not sufficient for studying the convergence of Schwarz scheme in \cite{Shin2019Parallel} because (i) a terminal perturbation is missing and (ii) the dual sensitivity is not formally analyzed.
}

This paper extends the related literature \cite{Na2020Exponential, Barrows2014Time, Xu2018Exponentially, Shin2019Parallel} in the following aspects. (i) {\red We expand the results in \cite{Na2020Exponential} by enabling a terminal perturbation, and more importantly, complement \cite{Na2020Exponential} by showing that EDS also holds for the dual solution of nonlinear OCPs.~We emphasize that obtaining dual sensitivity from primal sensitivity is not straightforward; and we emphasize that the former has not been studied even in the context of convex OCPs. To address this knowledge gap, we delve deeper into the convexification technique in \cite{Na2020Exponential}; and show that, although the dual solution is altered by convexification (not preserved like primal solution), it is shifted only by an affine transformation of the primal solution (cf. Theorem \ref{thm:1}). With~this relation, we further provide a \textit{stagewise} closed form of the dual solution (cf. Theorem \ref{thm:2}) and establish dual sensitivity (cf. Theorem \ref{thm:dual:sensitivity}). (ii) By sensitivity analysis, we enhance the~existing overlapping decomposition and Schwarz schemes \cite{Barrows2014Time, Xu2018Exponentially, Shin2019Parallel} by providing a convergence analysis for} time-varying nonlinear OCPs, which cover a much wider range of applications {\red than convex OCPs in \cite{Barrows2014Time,Xu2018Exponentially} and time-invariant OCPs in \cite{Shin2019Parallel}. Furthermore, our primal-dual sensitivity analysis validates the conjecture~in~\cite{Shin2019Parallel}.}  (iii) We prove that the overlapping Schwarz scheme enjoys \textit{linear} convergence locally, provided the overlap size $\tau$ is sufficiently large. We also show that the linear rate is given by $C\rho^\tau$, where $C>0$, $\rho\in(0,1)$ are constants independent of horizon length $N$. In other words, the linear rate improves {\em exponentially} with the overlap size. {\red As a special case, we also show that the Schwarz scheme exhibits global convergence for a linear-quadratic OCP setting (but potentially with nonconvex objective)}. This result is of relevance, as it suggests that the Schwarz method can be used to solve quadratic programs and linear algebra systems inside second-order algorithms such as sequential quadratic programming and interior-point methods. {\red Such a special case is still more general than \cite{Xu2018Exponentially} and requires a fundamentally different proof technique.} Our theory explains favorable performance noticed in recent computational studies that use this approach \cite{Shin2020Graph}.

{\red It is worth mentioning that a recent work \cite{Na2020Superconvergence} made use of the established primal-dual sensitivity in this paper to study a real-time online model predictive control algorithm. Although this paper also solves nonlinear OCPs, there are significant differences in problem setup, techniques, and results with  \cite{Na2020Superconvergence}. First, \cite{Na2020Superconvergence} solved~\eqref{pro:1} in an {\it online} fashion, where a single Newton step is performed~to solve the subproblem {\it inexactly}, and then the system shifts to the next stage~with a \textit{new} subproblem to be targeted. Online algorithms are a special class of inexact methods for nonlinear predictive control problems, mostly used for systems that require a fast reaction to disturbances (for example, autonomous vehicles) \cite{Ohtsuka2004continuation/GMRES,Diehl2005Real,Zavala2009advanced,Zavala2010Real}. In contrast, our approach solves a long-horizon problem~\eqref{pro:1} in an {\it offline} fashion with a parallel environment, where problems do not shift but are solved {\it to the optimality}. Second, \cite{Na2020Superconvergence} relied on the sensitivity (of linear-quadratic OCPs) to show a decay structure of KKT matrix inverse, based on which \cite{Na2020Superconvergence} explored Newton's method and showed a linear-quadratic error recursion. In contrast, we rely on the sensitivity (of nonlinear OCPs) to have an increasingly more accurate boundary primal-dual iterates for subproblems and, hence, the subproblem solutions are increasingly closer to the \textit{truncated full-horizon solution}. Third, \cite{Na2020Superconvergence} only showed the real-time iterates stably track the solution (i.e. stay in a neighborhood), while we show the offline iterates converge to the solution linearly with a quantitative relation between the convergence rate and the overlap~size.
}

Our work focuses on the convergence properties of overlapping Schwarz scheme, which is a new and different paradigm for decomposing OCPs {\red compared to traditional approaches \cite{Beccuti2004Temporal, Boyd2010Distributed, Kumar2018Stochastic, Zavala2016New}}. This approach is interesting in that it spans a spectrum of algorithms that go from a fully centralized/sequential communication pattern (the overlap is the entire horizon) to a no-interaction communication pattern (no overlap). This iterative approach thus provides flexibility to enable different hardware implementations. The paper is motivated~by the great success observed in practice \cite{Barrows2014Time, Xu2018Exponentially, Shin2019Parallel, Shin2020Graph, Collet2020Non}. Prior to this work, the convergence of overlapping decomposition schemes has only been explored for {\red restrictive} linear-quadratic convex cases. Here, we show that the Schwarz decomposition exhibits linear convergence for general nonconvex OCPs (cf. Theorem \ref{thm:conv}).~This provides an advantage over the widely-used ADMM, whose convergence {\red is established mostly for restrictive setups that do not apply for \eqref{pro:1}. For instance, the standard result only deals with convex problems \cite{Boyd2010Distributed}, and the nonconvex results in  \cite[Equation (2.2)]{Hong2016Convergence} and \cite[Equation (1)]{Wang2018Global} do not allow nonlinear dynamical constraints as in \eqref{pro:1b}.} Furthermore, we numerically demonstrate that overlapping Schwarz has much faster convergence than ADMM and may be as efficient as Ipopt (a centralized solver). This observation is important because Ipopt is highly efficient but does not offer flexibility in hardware implementations. Establishing convergence theory for Schwarz scheme is also meaningful from a control practitioner's stand-point, as it explains the performance observed in many recent computational studies. Moreover, the established theory provides insights on how the scheme will behave when tuning the overlap size $\tau$. Our primal-dual EDS result also provides a foundation for analyzing the behavior of a variety of algorithms and approximations for predictive control \cite{Shin2020Diffusing, Gruene2020Efficient, Gruene2021Abstract}.

The remainder of the paper is organized as follows.  In Section \ref{sec:2} we establish primal-dual sensitivity results for \eqref{pro:1}. In Section \ref{sec:4} we describe the overlapping Schwarz scheme and its convergence analysis. Numerical results are shown in Section \ref{sec:6} and  conclusions are presented in Section \ref{sec:7}.

\section{Primal-Dual Exponential Decay of Sensitivity}\label{sec:2}

In this section, we {\red enhance the analysis in \cite{Na2020Exponential}} and establish~a primal-dual sensitivity result for nonlinear OCPs that we call exponential decay of sensitivity (EDS). We use the following notation: for $n, m\in\mZ_{>0}$, we let $[n, m]$, $[n, m)$, $(n, m]$, $(n, m)$ be the corresponding integer sets; also, $[n] = [0,n]$. Boldface symbols denote column vectors. For a set of vectors $\{\ba_i\}_{i=m}^n$, $\ba_{m:n} = (\ba_m; \ldots; \ba_n)$ represents a long vector that stacks~them together. For scalars $a, b$, $a\vee b = \max(a, b)$; $a\wedge b = \min(a, b)$. For a set of matrices $\{A_i\}_{i=m}^n$, $\prod_{i=m}^nA_i = A_nA_{n-1}\cdots A_m$ if $m\leq n$ and $I$ otherwise. Without specification, $\|\cdot\|$ denotes either $\ell_2$ norm for a vector or {\red spectral} norm for a matrix. For a function $f:\mR^n\rightarrow\mR^m$, $\nabla f\in\mR^{n\times m}$ is its Jacobian.

\subsection{{\red Sensitivity Analysis and} Primal EDS Results}

We begin by analyzing the sensitivity of the primal solution. Most of the results in this subsection are presented in \cite{Na2020Exponential}, but we revisit them for completeness and to lay the groundwork for the new dual sensitivity in Section \ref{sec:3}. We rewrite \eqref{pro:1} by explicitly expressing the dependence on external {\em data} (parameters) as
\begin{subequations}\label{pro:2}
\begin{align}
\min_{\substack{\{\bx_k\}\\ \{\bu_k\}}}\  & \sum_{k=0}^{N-1} g_k(\bx_k, \bu_k; \bd_k) + g_N(\bx_N; \bd_N), \label{pro:2a}\\
\text{s.t.\ \ } & \bx_{k+1} = f_k(\bx_k, \bu_k; \bd_k), \; k\in[N-1],\; (\blambda_k) \label{pro:2b}\\
& \bx_0 = \bar{\bx}_0,\;(\blambda_{-1}). \label{pro:2c}
\end{align}
\end{subequations}
Here $\bd_k\in\mR^{n_d}$ and $\bd_{-1} = \bbx_0$ are the external problem data. In what follows, we let $\bz_k = (\bx_k; \bu_k)$, $\bw_k=(\bz_k;\blambda_k)$ for $k\in[N-1]$, and $\bw_N = \bz_N = \bx_N$ and $\bw_{-1}=\blambda_{-1}$. $\bx = \bx_{0:N}$ (similar for $\bu, \bd, \blambda, \bz, \bw$) is the full vector with variables being ordered by stages. We also denote $\bz=(\bx,\bu)$, $\bw = (\bx, \bu,\blambda)$ for simplicity. We let $n_{\bx}$ (similar for $n_{\bu}, n_{\bd}, n_{\blambda}, n_{\bz}, n_{\bw}$) be the dimension of $\bx$.

The Lagrange function of \eqref{pro:2} is
\begin{align*} %\label{equ:Larg:NLDP}
\mL(\bw; \bd) =& \sum_{k=0}^{N-1}\overbrace{g_k(\bz_k; \bd_k) + \blambda_{k-1}^T\bx_k - \blambda_k^Tf_k(\bz_k; \bd_k)}^{ \mL_k(\bz_k,\blambda_{k-1:k}; \bd_k)}\nonumber\\
& +\underbrace{ g_N(\bz_N; \bd_N) + \blambda_{N-1}^T\bx_N}_{\mL_N(\bz_N, \blambda_{N-1}; \bd_N)} - \blambda_{-1}^T\bd_{-1}.
\end{align*}
Suppose that $\tw(\bd) = (\tx(\bd), \tu(\bd), \tlambda(\bd))$ is a local~minimizer of \eqref{pro:2} with unperturbed data $\bd$. Sensitivity analysis~characterizes how the solution trajectory $\tw(\bd)$ varies with respect to perturbations on $\bd$. In particular, we let $\bl \in \mR^{n_{\bd}}$ be the perturbation direction of $\bd$ and let the corresponding parametric perturbation path be:
\begin{equation}\label{equ:perturb:path}
\bd(h, \bl) = \bd + h \bl + o(h). 
\end{equation}
Then we define directional derivatives of {\red solution trajectories}~as
\begin{subequations}\label{equ:direc:deriva}
\begin{align}
\tp_k = & \lim\limits_{h\searrow 0}\frac{\tx_k(\bd(h,\bl)) - \tx_k(\bd)}{h}, \text{\ }\forall k\in[N],\\ 
\tq_k = & \lim\limits_{h\searrow 0}\frac{\tu_k(\bd(h,\bl)) - \tu_k(\bd)}{h}, \text{\ } \forall k\in[N-1],\\
\tzeta_k = & \lim\limits_{h\searrow 0}\frac{\tlambda_k(\bd(h,\bl)) - \tlambda_k(\bd)}{h}, \text{\ } \forall k\in[-1, N-1].
\end{align}
\end{subequations}
Sensitivity analysis is equivalent to bounding the magnitude of directional derivatives. We are particularly interested in~bounding $\|\tp_k\|$, $\|\tq_k\|$, $\|\tzeta_k\|$ when only $\bd_i$ is perturbed. That is~we enforce $\bl = \be_i$, where $\be_i\in\mR^{n_{\bd}}$ for $i\in[-1, N]$ is any~unit vector with support within stage $i$.

\begin{definition}[Reduced Hessian]\label{def:1}

For $k\in[N-1]$, we let $A_k = \nabla_{\bx_k}^Tf_k(\bz_k; \bd_k)$, $B_k = \nabla_{\bu_k}^Tf_k(\bz_k; \bd_k)$, $C_k = \nabla_{\bd_k}^Tf_k(\bz_k; \bd_k)$, and Hessian matrices be
\begin{align*}
H_k(\bw_k; \bd_k) =& \begin{pmatrix}
Q_k & S_k^T\\
S_k & R_k
\end{pmatrix} = \begin{pmatrix}
\nabla_{\bx_k}^2\mL_k & \nabla_{\bx_k\bu_k}^2\mL_k\\
\nabla_{\bu_k\bx_k}^2\mL_k & \nabla_{\bu_k}^2\mL_k
\end{pmatrix},\\
D_k(\bw_k; \bd_k) = & \begin{pmatrix}
D_{k1} & D_{k2}
\end{pmatrix} = \begin{pmatrix}
\nabla_{\bd_k\bx_k}^2\mL_k & \nabla_{\bd_k\bu_k}^2\mL_k
\end{pmatrix},
\end{align*}
together with $H_N(\bz_N; \bd_N) = \nabla_{\bx_N}^2\mL_N(\bz_N,\blambda_{N-1}; \bd_N)$ and $D_N(\bz_N; \bd_N) = \nabla_{\bd_N\bx_N}^2\mL_N(\bz_N,\blambda_{N-1}; \bd_N)$. The evaluation point of $A_k, B_k, C_k$ is suppressed for conciseness. We also use $Q_N$ and $H_N$ interchangeably. In addition, we let $H(\bw; \bd) = \diag(H_0, \ldots, H_N)\in\mR^{n_{\bz}\times n_{\bz}}$ and let Jacobian matrix $G(\bz; \bd)\in\mR^{n_{\bx}\times n_{\bz}}$ (which has full row rank) be
\begin{equation*}
\left(\begin{smallmatrix}
I\\
-A_0 & -B_0 & I\\
&&-A_1 & -B_1 & I\\
&&&&\ddots & \ddots\\
&&&&&-A_{N-1} & -B_{N-1} & I
\end{smallmatrix}\right).
\end{equation*}
Let $Z(\bz; \bd)\in\mR^{n_{\bz}\times n_{\bu}}$ ($n_{\bu} = n_{\bz} - n_{\bx}$) be a full column rank matrix whose columns are orthonormal and span the null space of $G(\bz; \bd)$. Then the reduced Hessian is {\red defined as}
\begin{equation*}
ReH(\bw; \bd) = Z^THZ.
\end{equation*}
\end{definition}

We then introduce three assumptions to establish~sensitivity: uniform second-order sufficient condition (SOSC), controllability, and boundedness. Recall that $\bd$ is the unperturbed data with $\tw(\bd)$ being a local solution. We drop $\bd$ hereinafter from the notation and denote the solution as $\tw$.

\begin{assum}[Uniform SOSC]\label{ass:1}
At $(\tw;\bd)$, the reduced Hessian of \eqref{pro:2} satisfies $ReH(\tw; \bd)\succeq \gamma_H I$ for some uniform constant $\gamma_H>0$ independent of horizon $N$.
\end{assum}

{\red Assumption \ref{ass:1}} requires the Lagrangian Hessian to be positive definite in the null space of the linearized constraints (instead of in the whole space). The uniformity in Assumption \ref{ass:1} means the independence of $\gamma_H$ from $N$. 

\begin{definition}[Controllability Matrix]
For any $k\in[N-1]$ and evolution length $t\in[1, N-k]$, the controllability matrix is given by
\begin{align*}
\Xi_{k, t}(&\bz_{k:k+t-1}; \bd_{k:k+t-1}) \\
= & \bigl(\begin{smallmatrix}
B_{k+t-1} & A_{k+t-1}B_{k+t-2} & \ldots & \rbr{\prod_{l=1}^{t-1}A_{k+l}}B_k
\end{smallmatrix}\bigr)\in\mR^{n_x\times tn_u},
\end{align*}
where $\{A_i\}_{i=k+1}^{k+t-1}$, $\{B_i\}_{i=k}^{k+t-1}$ evaluates at $\{(\bz_i; \bd_i)\}_{i=k}^{k+t-1}$.
\end{definition}

\begin{assum}[Uniform Controllability]\label{ass:2}
At $(\tz; \bd)$, there exist {\red uniform} constants $\gamma_C, t>0$ independent of $N$ such that $\forall k\in[N-t]$, $\exists1\leq t_k\leq t$ and such that
\begin{equation*}
\Xi_{k, t_k}\Xi_{k, t_k}^T\succeq\gamma_C I ,
\end{equation*} 
where $\Xi_{k, t_k}$ evaluates at $(\tz_{k:k+t_k-1}; \bd_{k:k+t_k-1})$.
\end{assum}

The controllability condition is imposed on the constraint matrices. It captures the \textit{local geometry} of the null space,~which follows the notion of uniform complete controllability, introduced in \cite[Definition 3.1]{Keerthi1988Optimal} and used in sensitivity analysis in \cite[Definition 2.2]{Xu2018Exponentially}.

\begin{assum}[Uniform Boundedness]\label{ass:3}
At $(\tw;\bd)$, there~exists {\red a uniform} constant $\Upsilon_{\text{upper}}$ independent of $N$ such that $\|H_N\|\leq \Upsilon_{\text{upper}}$ and $\forall k\in [N-1]$:
\begin{equation*}
\|H_k\| \vee \|D_k\| \vee \|A_k\| \vee \|B_k\| \vee\|C_k\|\leq \Upsilon_{\text{upper}}.
\end{equation*}
\end{assum}

The following result shows that $\tp_k$, $\tq_k$, $\tzeta_k$ in \eqref{equ:direc:deriva} are the solution of a linear-quadratic OCP provided SOSC holds at~$\tw$. 

\begin{thm}[Sensitivity of Problem \eqref{pro:2}]\label{thm:sen:LQP}

Consider OCP \eqref{pro:2}, and suppose $\bd$ is perturbed along the path \eqref{equ:perturb:path}. If $\tw$ satisfies SOSC, the directional derivatives $(\tp_k, \tq_k, \tzeta_k)$ defined in \eqref{equ:direc:deriva} exists and is the primal-dual solution of the problem:
\begin{subequations}\label{pro:3}
\begin{align}
&\min_{\substack{\{\bp_k\}\\\{\bq_k\}}} \ \sum_{k=0}^{N-1}\biggl(\begin{smallmatrix}
\bp_k\\
\bq_k\\
\bl_k
\end{smallmatrix}\biggr)^T\biggl(\begin{smallmatrix}
Q_k &S_k^T & D_{k1}^T\\
S_k & R_k & D_{k2}^T\\
D_{k1} & D_{k2} & \0
\end{smallmatrix}\biggr)\biggl(\begin{smallmatrix}
\bp_k\\
\bq_k\\
\bl_k
\end{smallmatrix}\biggr) \nonumber\\
&\qquad\qquad \quad\quad\quad + \biggl(\begin{smallmatrix}
\bp_N\\
\bl_N
\end{smallmatrix}\biggr)^T\biggl(\begin{smallmatrix}
Q_N & D_N^T\\
D_N & \0
\end{smallmatrix}\biggr)\biggl(\begin{smallmatrix}
\bp_N\\
\bl_N
\end{smallmatrix}\biggr), \label{pro:3a}\\
&\st\ \bp_{k+1} = A_k\bp_k + B_k\bq_k + C_k\bl_k,\; (\bzeta_k)\label{pro:3b}\\
&\qquad\bp_0 = \bl_{-1},\; (\bzeta_{-1}). \label{pro:3c}
\end{align}
\end{subequations}
Here, $\bzeta_{-1:N-1}$ are dual variables associated to constraints. All matrices are evaluated at $(\tw;\bd)$.
\end{thm}

\begin{proof}
See {\red \cite[Theorem 5.61]{Bonnans2000Perturbation}} for the proof. Observe from the structure of $G(\bz;\bd)$ in Definition \ref{def:1} that the linear independence constraint qualification (LICQ) holds with any $(\bz; \bd)$. Thus,~the results hold for any perturbation direction $\bl$. 
\end{proof}

Let $\bxi_{k}=(\bp_k;\bq_k;\bzeta_k)$ for $k\in[N-1]$, and $\bxi_{-1}=\bzeta_{-1}$ and $\bxi_N=\bp_N$. Further, $\bxi=(\bp, \bq,\bzeta)=\bxi_{-1:N}$ (similar for $\bp, \bq, \bzeta$) is the full vector with variables being ordered by stages. From SOSC (cf. Assumption \ref{ass:1}), LICQ, and \cite[Lemma 16.1]{Nocedal2006Numerical}, we know that $\txi=(\tp,\tq,\tzeta)$ is unique global solution of \eqref{pro:3}. However, the indefiniteness of the Hessians $H_k$ in Problem \eqref{pro:3} brings difficulty in analyzing the solution of \eqref{pro:3} obtained from the Riccati recursion. Thus, \cite{Na2020Exponential} relied on the convexification procedure proposed in \cite{Verschueren2017Sparsity}, which transfers \eqref{pro:3} into another linear-quadratic program whose new matrices $\tilde{H}_k$ are positive definite. The procedure is displayed in Algorithm \ref{alg:convex:proce}. One inputs quadratic matrices $\{H_k, D_k, A_k, B_k, C_k\}$ of Problem \eqref{pro:3}, and then obtains new matrices $\{\tilde{H}_k, \tilde{D}_k\}$. The constraint matrices $\{A_k, B_k, C_k\}$ need not be transformed. As shown in \cite{Na2020Exponential}, with~a proper set of $\beta>0$, Algorithm~\ref{alg:convex:proce} preserves the primal solution. We will show later that Algorithm \ref{alg:convex:proce} shifts the dual solution.
  
%Intuitively, this procedure convexifies Hessians by recursively adding and subtracting quadratic terms, which add up to a constant on the null space of \eqref{pro:3b}-\eqref{pro:3c}. 

\begin{thm}[Primal EDS]\label{thm:primal:sensitivity}

Let Assumptions \ref{ass:1}, \ref{ass:2}, \ref{ass:3} hold at the solution $\tw$ of Problem \eqref{pro:2}. Then there exist constants $\Upsilon>0$, $\rho\in(0, 1)$, which only depend on constants in the assumptions and hence are independent of horizon length $N$, such that
\begin{enumerate}[label=(\alph*),topsep=-0.2cm]
\setlength\itemsep{0.2em}
\item if $\bl = \be_i$, $\forall i\in[N]$, then $\|\tp_k\| \vee \|\tq_k\| \leq \Upsilon\rho^{|k-i|}$ for $k\in[N-1]$ and $\|\tp_N\|\leq \Upsilon\rho^{N-i}$;
\item if $\bl = \be_{-1}$, then $\|\tp_k\| \vee \|\tq_k\| \leq \Upsilon\rho^{k}$ for $k\in[N-1]$ and $\|\tp_N\|\leq \Upsilon \rho^{N}$.
\end{enumerate}
\end{thm}

This is \cite[Theorem 5.7]{Na2020Exponential}\footnote{
We note that Problem \eqref{pro:2} is slightly different from the one in \cite{Na2020Exponential}, in that \cite{Na2020Exponential} does not include the terminal data $\bd_N$. However, {\red with fairly slight adjustment, their Theorem 5.7 can be extended to the case $\bl = \be_N$.}} 
and indicates that the impact of a perturbation on $\bd_i$ on the primal solution $\bz^\star_k$ at stage $k$ decays exponentially fast as one moves away from stage $i$.

% considered  Specifically, the problem in \cite{Na2020Exponential} does not include the terminal data $\bd_N$. However, by doing slight modifications in (3.4), (3.5), and Lemma 5.1 in \cite{Na2020Exponential} (specifically, replacing $\sum_{i=k+1}^{N-1}(M_i^{k+1})^T\bl_i$ by $\sum_{i=k+1}^{N}(M_i^{k+1})^T\bl_i$, $\sum_{i=k+1}^{N-1}\bl_i^TM_i^k\bp_k$ by $\sum_{i=k+1}^{N}\bl_i^TM_i^k\bp_k$, and $\sum_{i=0}^{N-1}U_i^k\bl_i$ by $\sum_{i=0}^{N}U_i^k\bl_i$ in their notations at the mentioned points in \cite{Na2020Exponential}, and using $M_N^N = D_{N1}$ (i.e., $D_N$ in our paper)), all conclusions can be extended to the case $\bl = \be_N$ as well. Adding the perturbation on $\bd_N$ is necessary to establish convergence of the overlapping Schwarz scheme. 

\subsection{Dual EDS Results}\label{sec:3}

We now present dual sensitivity for \eqref{pro:2} based on convexification procedure in Algorithm \ref{alg:convex:proce}. As shown in \cite{Na2020Exponential}, because~of the positive definiteness of $\tilde{H}_k$, the convexified problem (obtained by replacing $\{H_k, D_k\}$ in \eqref{pro:3} with outputs $\{\tilde{H}_k, \tilde{D}_k\}$) also has a unique global solution. Thus, we need to understand how~the dual solutions are affected by convexification. We will show from the Karush-Kuhn-Tucker (KKT) conditions (i.e.,~the first-order necessary conditions) that Algorithm \ref{alg:convex:proce} shifts the~dual solution by {\red an affine} transformation of the primal solution. In what follows, we use $\mL\mQ\mP$ to denote Problem \eqref{pro:3} defined with original matrices $\{H_k, D_k\}$, and $\mC\mL\mQ\mP$ to denote Problem \eqref{pro:3} defined with convexified matrices $\{\tilde{H}_k, \tilde{D}_k\}$. Furthermore, $\txic=(\tpc, \tqc, \tzetac)$ denotes the (global) primal-dual solution of $\mC\mL\mQ\mP$. Recall that, by Theorem \ref{thm:sen:LQP}, $(\tp, \tq, \tzeta)$ is the global solution of $\mL\mQ\mP$.

The following result establishes a relationship between the solutions {\red $(\tp, \tq, \tzeta)$ and $(\tpc, \tqc, \tzetac)$}.

\begin{algorithm}[t]\caption{Convexification Procedure}\label{alg:convex:proce}
\begin{algorithmic}[1]
\STATE \textbf{Input:} $\{H_k, D_k\}_{k=0}^N$, $\{A_k, B_k, C_k\}_{k = 0}^{N-1}$, $\beta>0$;
\STATE $\tilde{H}_N = \tilde{Q}_ N = \beta I$;
\STATE $\bar{Q}_N = Q_N - \tilde{Q}_N$;
\FOR {$k = N-1, \ldots, 0$}
\STATE \hskip-3cm$\substack{\biggl(\begin{smallmatrix}
\hat{Q}_k & \tilde{S}_k^T & \tilde{D}_{k1}^T\\
\tilde{S}_k & \tilde{R}_k & \tilde{D}_{k2}^T\\
\tilde{D}_{k1} & \tilde{D}_{k2} & *
\end{smallmatrix}\biggr) = \biggl(\begin{smallmatrix}
Q_k & S_k^T & D_{k1}^T\\
S_k & R_k & D_{k2}^T \\
D_{k1} & D_{k2} & \0
\end{smallmatrix}\biggr) \\ \qquad\qquad\qquad\qquad\qquad\qquad\qquad\qquad\qquad\qquad+ \biggl(\begin{smallmatrix}
A_k^T\\
B_k^T\\
C_k^T
\end{smallmatrix}\biggr)\bar{Q}_{k+1}(\begin{smallmatrix}
A_k & B_k & C_k
\end{smallmatrix})}$
\STATE $\tilde{Q}_k = \tilde{S}_k^T\tilde{R}_k^{-1}\tilde{S}_k + \beta I$
\STATE $\tilde{H}_k = \begin{pmatrix}
\tilde{Q}_k & \tilde{S}_k^T\\
\tilde{S}_k & \tilde{R}_k
\end{pmatrix}$
\STATE $\bar{Q}_k = \hat{Q}_k - \tilde{Q}_k$;
\ENDFOR
\STATE \textbf{Output:} $\{\tilde{H}_k\}_{k = 0}^N$, $\{\tilde{D}_k\}_{k = 0}^{N-1}$, $D_N (= \tilde{D}_N)$.
\end{algorithmic}
\end{algorithm}

\begin{thm}\label{thm:1}
Under Assumption \ref{ass:1}, we execute Algorithm \ref{alg:convex:proce} with $\beta\in(0, \gamma_H)$ for $\mL\mQ\mP$. We then have that
\begin{align}\label{equ:relation}
\tp = \tpc, \quad \tq = \tqc, \quad \tzeta = \tzetac - 2\bar{Q}\tp,
\end{align}
where $\bar{Q} = \diag(\bar{Q}_0,\ldots, \bar{Q}_N)$ with $\{\bar{Q}_k\}_{k=0}^N$ is defined in Algorithm \ref{alg:convex:proce} recursively. 

\end{thm}

\begin{proof}
See Appendix \ref{pf:thm:1}.
\end{proof}

Using \eqref{equ:relation}, we first focus on $\mC\mL\mQ\mP$ and establish the~exponential decay result for $\tzetac$. Then we use relation \eqref{equ:relation} to bound $\tzeta$. {\red The motivation is that some nice properties of $\{\tilde{H}_k, \tilde{D}_k\}$ does not hold for $\{H_k, D_k\}$, which brings difficulties to directly study $\mL\mQ\mP$.}

The following theorem provides the \textit{stagewise} closed form of the dual solution for linear-quadratic problems (either $\mL\mQ\mP$ or $\mC\mL\mQ\mP$). {\red Our notation is the same as \cite[Lemma 3.5]{Na2020Exponential}, which provides the stagewise closed form of the primal solution.}

\begin{thm}\label{thm:2} 
	
Consider~$\mL\mQ\mP$ under Assumption \ref{ass:1}. Suppose $(\tp, \tq)$ is the primal solution. Then the dual solution $\tzeta$ at each stage is:
\begin{align}\label{equ:dual:form}
\tzeta_k =& -2K_{k+1}\tp_{k+1} + 2\sum_{i=k+1}^{N}(M_i^{k+1})^T\bl_i \nonumber\\
&  + 2\sum_{i=k+1}^{N-1}(V_i^{k+1})^TC_i\bl_i, \quad \forall k\in[-1,N-1],
\end{align}
with $K_N = Q_N$, $D_{N1} = D_N$, $D_{N2} = \0$, and $\forall k\in[N-1]$,
\begin{align*}
W_k = &R_k + B_k^TK_{k+1}B_k,\\
K_k = & - (B_k^TK_{k+1}A_k + S_k)^T W_k^{-1}(B_k^TK_{k+1}A_k + S_k)\\
& +Q_k + A_k^TK_{k+1}A_k,\\
P_k = & -W_k^{-1}(B_k^TK_{k+1}A_k + S_k),\\
E_k =& A_k + B_kP_k,\\
V_i^k =& -K_{i+1}\prod_{j=k}^iE_j, \hskip1.7cm \forall i\in[N-1],\\
M_i^k =& -(D_{i1} + D_{i2}P_i)\prod_{j=k}^{i-1}E_j, \quad \forall i\in[N].
\end{align*}
We obtain a similar formula for $\tzetac$ of $\mC\mL\mQ\mP$, where one replaces $\{H_k, D_k\}$ in the above recursions by $\{\tilde{H}_k, \tilde{D}_k\}$.
\end{thm}

\begin{proof}
See Appendix \ref{pf:thm:2}. 
\end{proof}

We now study the dual solution $\tzetac$ of $\mC\mL\mQ\mP$. To enable concise notation, we abuse the notation $K_k, M_i^k, V_i^k$, and so on to denote the matrices computed by $\{\tilde{H}_k, \tilde{D}_k\}$.  The following lemma establishes the exponential decay for $\tzetac$.

\begin{lem}\label{lem:dual}
Let Assumptions \ref{ass:1}, \ref{ass:2}, \ref{ass:3} hold at the unperturbed solution $\tw$ of Problem \eqref{pro:2}. We execute Algorithm \ref{alg:convex:proce} with $\beta\in(0, \gamma_H)$. Let $\tzetac$ be the optimal dual solution of $\mC\mL\mQ\mP$. Then there exist constants $\Upsilon'>0$, $\rho\in(0,1)$, independent of $N$, such that for any $k\in[-1, N-1]$,
\begin{enumerate}[label=(\alph*),topsep=-0.9cm]
\setlength\itemsep{0.2em}
\item if $\bl = \be_i$ for $i\in[N]$, then $\|\tzetac_k\| \leq \Upsilon'\rho^{|k+1-i|}$;
\item if $\bl = \be_{-1}$, then $\|\tzetac_k\| \leq \Upsilon'\rho^{k+1}$.
\end{enumerate}
\end{lem}

\begin{proof}
See Appendix \ref{pf:lem:dual}.
\end{proof}

We note that the constant $\rho$ in this result is the same as the one used in Theorem \ref{thm:primal:sensitivity}. Combining Lemma \ref{lem:dual} with Theorem~\ref{thm:1}, we can finally bound the dual solution for $\mL\mQ\mP$.

\begin{thm}[Dual EDS]\label{thm:dual:sensitivity}

Let Assumptions \ref{ass:1}, \ref{ass:2}, \ref{ass:3} hold at the unperturbed solution $\tw$ of Problem \eqref{pro:2}. Then for any $k\in[-1, N-1]$, Lemma \ref{lem:dual} holds for $\tzeta_k$ with some constants $\Upsilon''>0$, $\rho\in(0,1)$ that are determined by constants in the assumptions and hence are independent of $N$.
\end{thm}

\begin{proof}
See Appendix \ref{pf:thm:dual:sensitivity}.
\end{proof}

Combining Theorems \ref{thm:primal:sensitivity} and \ref{thm:dual:sensitivity}, we get the desired primal-dual EDS result. The perturbation on the left and right boundaries $\{-1,N\}$ are of particular interest in the following sections. Redefining $\Upsilon\leftarrow {\red \sqrt{3}}\max(\Upsilon,\Upsilon''\rho^{-1})$ yields the following\footnote{\red This is because $\|\txi_k\| \leq \sqrt{3}(\|\tp_k\|\vee \|\tq_k\|\vee \|\tzeta_k\|)$. One can apply Theorem \ref{thm:primal:sensitivity} for bounding $\|\tp_k\|\vee \|\tq_k\|$ and Theorem \ref{thm:dual:sensitivity} for bounding $\|\tzeta_k\|$.}:
\begin{enumerate}[label=(\alph*),topsep=-0cm]
\setlength\itemsep{0.3em}
\item if $\bl = \be_N$, then $\|\txi_k\| \leq \Upsilon\rho^{N-k}$;
\item if $\bl = \be_{-1}$, then $\|\txi_k\| \leq \Upsilon\rho^{k}$,
\end{enumerate}
where $\txi_k = (\tp_k; \tq_k; \tzeta_k)$ is defined in \eqref{equ:direc:deriva}. The above exponential property plays a key role in the analysis of the algorithm.

\section{Overlapping Schwarz Decomposition}\label{sec:4}

In this section we introduce the overlapping Schwarz scheme and establish its convergence. 

\subsection{Setting}

The full horizon of Problem \eqref{pro:1} is $[N]$. Suppose $T$ is the number of short horizons and $\tau$ is the overlap size. Then we can~decompose $[N]$ into $T$ consecutive intervals as
\begin{equation*}
[N] = \bigcup_{i=0}^{T-1}[m_{i}, m_{i+1}],
\end{equation*}
where $m_0=0<m_1<\ldots<m_{T}=N$. Moreover, we define the \textit{expanded} (overlapping) boundaries: 
\begin{equation}\label{equ:n1in2i}
n^1_i = (m_i - \tau) \vee 0, \quad n_i^2 = (m_{i+1} + \tau) \wedge N.
\end{equation}
Then we have $[N]=\bigcup_{i=0}^{T-1}[n_i^1,n_i^2]$ and 
\begin{equation*}
[m_i, m_{i+1}] \subset [n_i^1, n_i^2], \quad\quad \forall i\in[T-1].
\end{equation*}
In the overlapping Schwarz scheme, the truncated approximation within the interval $[m_i, m_{i+1}]$ is obtained by first solving a subproblem over an expanded interval $[n_i^1,n_i^2]$, then {\em discarding} the piece of the solution associated with the stages acquired from the expansion \eqref{equ:n1in2i}. We now introduce the subproblem for the expanded short horizon $[n_i^1,n_i^2]$. For any $i\in[T-1]$, the subproblem $\P_i$ for the interval $[n_i^1, n_i^2]$ is defined as
\begin{subequations}\label{pro:4}
\begin{align}
\min_{\substack{\{\bx_k\},\{\bu_{k}\}}} \;\;&
\sum_{k=n_i^1}^{n_i^2-1} g_k(\bx_k, \bu_k) + \tilde{g}_{n^2_i}(\bx_{n_i^2};\bbw_{n^2_i}),  \label{pro:4a}\\
\st\ \ & \bx_{k+1} = f_k(\bx_k, \bu_k), \;  k\in[n_i^1, n_i^2-1],\;(\blambda_k) \label{pro:4b}\\
& \bx_{n_i^1} = \bbx_{n_i^1},\quad (\blambda_{n^1_i-1}). \label{pro:4c}
\end{align}
\end{subequations}
Here, $\tilde{g}_{n^2_i}(\bx_{n_i^2};\bbw_{n^2_i})$ is the terminal cost function adjusted from $g_{n_i^2}(\bx_{n_i^2}, \bu_{n_i^2})$. It is parameterized by $\bbw_{n^2_i}=(\bbx_{n_i^2};\bbu_{n_i^2};\bblambda_{n_i^2})$ and is formally defined as
\begin{align*}
\tilde{g}_{n^2_i}(\bx_{n_i^2};\bbw_{n^2_i})= 
\begin{cases}
g_{n_i^2}(\bx_{n_i^2}, \bbu_{n_i^2}) -\bblambda_{n_i^2}^Tf_{n_i^2}(\bx_{n_i^2}, \bbu_{n_i^2})\\
\qquad+ \frac{\mu}{2}\|\bx_{n_i^2} - \bbx_{n_i^2}\|^2,\hskip 0.5cm i\in[T-2],\\
g_{N}(\bx_{N}),\hskip 2.6cm i=T-1,
\end{cases}  
\end{align*}
where $\mu$ is a uniform penalty parameter that {\em does not depend~on $i$}. In other words, $\mu$ is set uniformly over all subproblems. When $i\neq T-1$, the terminal cost is adjusted by a dual penalty and a quadratic penalty on the state. Intuitively, the dual penalty reduces the KKT residuals, while the quadratic penalty ensures that SOSC holds for subproblems provided it holds for the full problem and $\mu$ is set large enough~(see~Lemma~\ref{lem:mu}).~{\red The~formulation \eqref{pro:4} is adopted from \cite{Na2020Superconvergence}, which differs from the one in \cite{Shin2019Parallel}. In particular, \cite{Shin2019Parallel} imposed assumptions on subproblems, while we impose (standard) assumptions directly on the full~problem, which is more reasonable. An alternative~subproblem~formulation can be found in \cite[(5)]{Diehl2005Nominal}, where the terminal adjustment on the cost is replaced by a terminal constraint.}

We note that Problem \eqref{pro:4} is a parametric subproblem with $\P_i = \P_i(\bbx_{n_i^1}, \bbw_{n_i^2})$. The parameter $(\bbx_{n_i^1}, \bbw_{n_i^2})$ consists of the primal-dual data on both ends of the horizon (i.e. domain boundaries). Each time we have to first specify the parameter and then solve the subproblem. For $i = T-1$, $\bbw_{n^2_i}$ is not necessary (see the definition of $\tilde{g}_{N}(\cdot)$). The formal justification of the formulation in \eqref{pro:4} will be given in Lemma \ref{lem:2.5}.

\begin{definition}\label{def:3}

We define the following quantities for subproblem $\P_i$ with $i\in[T-1]$:
\begin{enumerate}[label=(\alph*),topsep=-0.9cm]
\setlength\itemsep{0.3em}
\item we let $\bw_{[i]}$ be the primal-dual variable of $\P_i$, i.e. $\bw_{[i]} = (\blambda_{n^1_i-1};\bw_{n^1_i:n^2_{i}-1};\bx_{n^2_{i}})$.

\item we let $\bw_{(i)}$ be the primal-dual variable of $\P_i$ on the~\textbf{non-overlapping} subdomains, i.e. $\bw_{(i)} = \bw_{m_i:m_{i+1}-1}$ (for~the boundaries, $\bw_{(0)}$ and $\bw_{(T-1)}$, are adjusted by letting $\bw_{(0)} = \bw_{-1:m_{1}-1}$ and $\bw_{(T-1)} = \bw_{m_{T-1}:m_{T}}$).

\item we let $\bw_{[-i]}$ be the parameter variable of $\P_i$, i.e. $\bw_{[-i]} = (\bx_{n^1_i};\bw_{n^2_{i}})$ (the boundary $\bw_{[-(T-1)]}$ is adjusted by letting $\bw_{[-(T-1)]} = \bx_{n_{T-1}^1}$).

\item we let $n_{[i]}$, $n_{(i)}$, $n_{[-i]}$ be the corresponding dimensions of $\bw_{[i]}, \bw_{(i)}, \bw_{[-i]}$.

\item we let $\bw^\dag_{[i]}(\cdot):\mR^{n_{[-i]}}\rightarrow \mR^{n_{[i]}}$ be the solution mapping of $\P_i$, i.e. $\bw^\dag_{[i]}(\bw_{[-i]})$ is a local solution of $\P_i(\bw_{[-i]})$.

\item for $k \in [n_i^1-1,n_i^2]$, we use $T_{k}(\bw_{[i]})$ to extract the variable on stage $k$ of $\bw_{[i]}$; we also use $T_{(i)}(\bw_{[i]})$ to extract variables of $\bw_{[i]}$ that are on non-overlapping subdomains.
\end{enumerate}
\end{definition}

The solution of $\mathcal{P}_i(\cdot)$ may not be unique. The issues of the existence and uniqueness of the solution will be resolved in Theorem \ref{thm:sol-conti}. For now, we assume that the solution $\bw_{[i]}^\dag(\cdot)$ exists and consider this as one of the local solutions.

We now formally present the overlapping Schwarz scheme~in Algorithm \ref{alg:decom:procedure}. Here, we use the superscript $(\cdot)^{(\ell)}$ to denote its value at the $\ell$-th iteration. In addition, we suppose that~the problem information (e.g. $\{f_k\}_{k=0}^{N-1}$, $\{g_k\}_{k=0}^{N}$) and the decomposition information (e.g. $\{m_i\}_{i=0}^T$ and $\{[n^1_i,n^2_i]\}_{i=0}^{T-1}$) are already given to the algorithm. Thus, the algorithm is well defined using only the initial guess $\bw^{(0)}$ of the full primal-dual solution as an input. Note that $\bx^{(0)}_{0}$ should match the initial state $\bbx_0$ given to the original problem \eqref{pro:1}.

\begin{algorithm}[!b]
\caption{Overlapping Schwarz Decomposition}
\label{alg:decom:procedure}
\begin{algorithmic}[1]
\STATE \textbf{Input:} $\bw^{(0)}$
\FOR {$\ell = 0,1,\ldots $}
\FOR {(in parallel) $i = 0,1,\ldots, T-1$}
\STATE $\bw_{(i)}^{(\ell +1)}=T_{(i)}({\bw}^\dag_{[i]}(\bw^{(\ell)}_{[-i]}))$;
\ENDFOR
\ENDFOR
\STATE \textbf{Output:} $\bw^{(\ell)}$
\end{algorithmic}
\end{algorithm}

Starting with $\bw^{(0)}$, the procedure iteratively finds the primal-dual solution $\bw^{(\ell)}$ for \eqref{pro:1}. At each iteration $\ell=0,1,\ldots$, the subproblems $\{\P_i(\bw_{[-i]}^{(\ell)})\}_{i=0}^{T-1}$ are solved to obtain the short-horizon solutions $\bw^\dag_{[i]}(\bw_{[-i]}^{(\ell)})$ over the expanded subdomains $\{[n^1_i,n^2_i]\}_{i=0}^{T-1}$. Here, we note that the previous primal-dual iterate enters into the subproblems as boundary conditions $\bw_{[-i]}^{(\ell)}=(\bx_{n^1_i}^{(\ell)};\bw_{n^2_i}^{(\ell)})$. This step is illustrated in Fig. \ref{fig:schematic}. The solution is then {\it restricted} to the non-overlapping subdomains $\{[m_i,m_{i+1}]\}_{i=0}^{T-1}$ by applying the operator $T_{(i)}(\cdot)$ (cf. Definition \ref{def:3}(f)), which is illustrated in Fig. \ref{fig:restrict}. After that, one concatenates the short-horizon solutions by $\bw^{(\ell+1)}=(\bw^{(\ell+1)}_{(0)};\ldots;\bw^{(\ell+1)}_{(T-1)})$. We do not explicitly write this step in Algorithm \ref{alg:decom:procedure} since updating the subvectors $\bw^{(\ell)}_{(i)}$ of $\bw^{(\ell)}$ over $i\in[T-1]$ effectively concatenates the short-horizon solutions. In Proposition \ref{prop:criteria} we provide stopping criteria for the scheme.

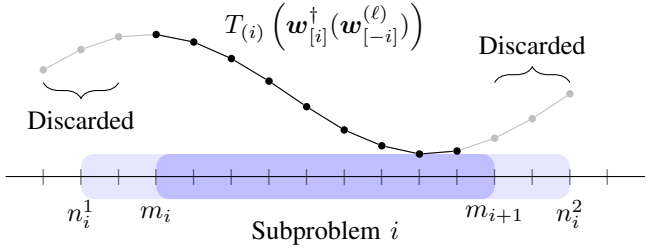
\begin{figure}[t]
  \centering
  \begin{tikzpicture}[scale=1, every node/.style={transform shape}]
    \fill[rounded corners=5pt,blue!10!white,line width=3pt](3.5,.2) rectangle ++(6.5,.6);
    \fill[rounded corners=5pt,blue!25!white,line width=3pt](4.5,.2) rectangle ++(4.5,.6);
    \foreach \x in {7,...,22}
    \node[scale=.6]  at (\x/2-1/2,.5) {$|$};
    \def\sh{0.6};
    \node[fill,circle,scale=.3,lightgray] (n1) at (3,1.32+\sh) {};
    \node[fill,circle,scale=.3,lightgray] (n2) at (3.5,1.59+\sh) {};
    \node[fill,circle,scale=.3,lightgray] (n3) at (4,1.76+\sh) {};
    \node[fill,circle,scale=.3] (n4) at (4.5,1.79+\sh) {};
    \node[fill,circle,scale=.3] (n5) at (5,1.69+\sh) {};
    \node[fill,circle,scale=.3] (n6) at (5.5,1.47+\sh) {};
    \node[fill,circle,scale=.3] (n7) at (6,1.17+\sh) {};
    \node[fill,circle,scale=.3] (n8) at (6.5,0.83+\sh) {};
    \node[fill,circle,scale=.3] (n9) at (7,0.52+\sh) {};
    \node[fill,circle,scale=.3] (n10) at (7.5,0.31+\sh) {};
    \node[fill,circle,scale=.3] (n11) at (8,.20+\sh) {};
    \node[fill,circle,scale=.3] (n12) at (8.5,.24+\sh) {};
    \node[fill,circle,scale=.3,lightgray] (n13) at (9,.41+\sh) {};
    \node[fill,circle,scale=.3,lightgray] (n14) at (9.5,.67+\sh) {};
    \node[fill,circle,scale=.3,lightgray] (n15) at (10,1+\sh) {};

    \draw[lightgray] (n1)--(n2)--(n3)--(n4)--(n5)--(n6)--(n7)--(n8)--(n9)--(n10)--(n11)--(n12)--(n13)--(n14)--(n15);
    \draw[black] (n4)--(n5)--(n6)--(n7)--(n8)--(n9)--(n10)--(n11)--(n12);
    \draw[-] (2.5,.5)--(11,.5);
    \node at (4.5,0) {$m_{i}$};
    \node at (9,0) {$m_{i+1}$};
    \node at (3.5,0) {$n^1_{i}$};
    \node at (10,0) {$n^2_{i}$};

    \draw [decorate,decoration={brace,amplitude=6pt}] (4,1.75)-- (3,1.75);
    \node at (3.5,1.25) {Discarded};
    \draw [decorate,decoration={brace,amplitude=6pt}] (9,1.75) -- (10,1.75);
    \node at (9.5,2.25) {Discarded};
    %% \draw [decorate,decoration={brace,amplitude=6pt}] (3,3) -- (10,3);
    %% \node at (6.5,3.5) {$\bw^\dag_{[i]}(\bw^{(\ell)}_{[-i]})$};
    
    \node at (6.75,2.5) {$T_{(i)}\rbr{\bw^\dag_{[i]}(\bw^{(\ell)}_{[-i]})}$};
    \node at (6.75,-0.25) {Subproblem $i$};
  \end{tikzpicture}
  \caption{Schematic of restriction operation.}\label{fig:restrict}
\end{figure}

Observe that, unless $\tau=0$, the boundary conditions $\bw_{[-i]}^{(\ell)}$ of subproblem $i$ are set by the output of other subproblems (in particular, adjacent ones if $\tau$ does not exceed the horizon lengths of the adjacent problems). Thus, the procedure aims to achieve coordination across the subproblems by exchanging the primal-dual solution information. Furthermore, one can observe that $n^1_i,n^2_i$ are at least $\tau$ stages apart from $[m_{i}, m_{i+1})$. This guarantees all iterates improve (as shown in the next section) because the adverse effect of misspecification of boundary conditions has enough stages to be damped. We can hence anticipate that having larger $\tau$ makes the convergence faster at the cost of having moderately larger subproblems.

We emphasize that Algorithm \ref{alg:decom:procedure} can be implemented in parallel, although the subproblems are coupled. This is because subproblems are parameterized by boundary variables $\{\bw_{[-i]} = (\bx_{n^1_i};\bw_{n^2_{i}})\}_i$ (as defined in Problem~\eqref{pro:4}) and, in each iteration, once $\{\bw_{[-i]}^{(\ell)}\}_i$ are specified all subproblems can be solved independently. To ensure convergence, we require information exchange between subproblems after each iteration. This is similar in spirit to traditional Jacobi and Gauss-Seidel schemes.

\subsection{Convergence Analysis}

We now establish convergence for Algorithm \ref{alg:decom:procedure}. A sketch~of convergence analysis is as follows. (i) We extend Assumptions \ref{ass:1}, \ref{ass:2}, \ref{ass:3} to a neighborhood of a local solution $\bw^\star$ of Problem \eqref{pro:1}. (ii) We show that $\bw^\dag_{[i]}(\tw_{[-i]}) = \tw_{[i]}$. (iii) We bound the difference of solutions $\bw^\dag_{[i]}(\bar{\bw}_{[-i]})$ and $\bw^\dag_{[i]}(\tw_{[-i]})$ by the difference in boundary conditions $\|\bar{\bw}_{[-i]} - \tw_{[-i]}\|$ using primal-dual EDS results in Section \ref{sec:2}. (iv) We combine (i)-(iii) and derive an explicit estimate of the local convergence rate.

\begin{definition}\label{def:2}
	
We define $\|(\cdot)\|_{\bw}$ as the stagewise max $\ell_2$-norm of the variable $(\cdot)$. For example,
\begin{align*}
\|\bz_k\|_{\bw} & =  \|\bx_k\| \vee \|\bu_k\|, \hskip1.2cm \|\bw_k\|_{\bw} = \|\bz_k\|_{\bw} \vee \|\blambda_k\|,\\
\|\bw \|_{\bw} & = \max_{k\in[-1,N]} \|\bw_k \|_{\bw}, \quad \; \|\bw_{[-i]} \|_{\bw} =  \|\bx_{n^1_i} \|\vee \|\bw_{n^2_i}\|_{\bw},\\
\|\bw_{[i]}\|_{\bw} & =  \|\blambda_{n^1_i-1}\|\vee\max_{k\in[n^1_i,n^2_i)} \|\bw_{k}\|_{\bw}\vee\|\bx_{n^2_i}\|.
\end{align*}
Further, we define $\N_{\varepsilon}(\cdot)$ as the (closed) $\varepsilon$-neighborhood of the variable $(\cdot)$ based on the norm $\|(\cdot)\|_{\bw}$. For example,
\begin{align*}
\N_{\varepsilon}(\tz_{k}) &= \{\bz_{k}\in\mR^{n_{\bz_k}}: \Vert \bz_{k}-\tz_{k}\Vert_{\bw} \leq \varepsilon \},\\
\N_{\varepsilon}(\tw_{k}) &= \{\bw_{k}\in\mR^{n_{\bw_k}}: \Vert \bw_{k}-\tw_k\Vert_{\bw} \leq \varepsilon \}.
\end{align*}
Similarly we have $\N_{\varepsilon}(\tw)$, $\N_{\varepsilon}(\tw_{[-i]})$, $\N_{\varepsilon}(\tw_{[i]})$.
\end{definition}

The norm $\|(\cdot)\|_{\bw}$ in Definition \ref{def:2} takes the maximum of the $\ell_2$-norms of stagewise state, control, and dual variables over the corresponding horizon. One can verify that it is indeed a vector norm (triangle inequality, absolute homogeneity, and positive definiteness hold). With Definition \ref{def:2}, here we extend assumptions in Section \ref{sec:2}, which are stated for a local solution point, to the {\it neighborhood} of such a local solution. In what follows, we inherit the notation in Definition \ref{def:1} but drop the reference variable $\bd$. We denote $A_k, B_k$ to be the Jacobian of $f_k(\bx_k, \bu_k)$ with respect to $\bx_k$ and $\bu_k$, respectively. $H_k$ is the Hessian of the Lagrange function with respect to $(\bx_k, \bu_k)$.

\begin{assum}\label{ass:4}

For a local primal-dual solution $\bw^\star$ of Problem \eqref{pro:1}, we assume that there exists $\varepsilon>0$ such that:
\begin{enumerate}[label=(\alph*),topsep=-0.9cm]
\setlength\itemsep{0.2em}
\item There exists {\red a uniform} constant $\gamma_H>0$ such that 
\begin{equation*}
ReH(\bw) \succeq \gamma_H I,
\end{equation*}
for any $\bw=\bw_{-1:N}$ with $\bw_k \in \N_\varepsilon(\tw_k)$ for $k\in [n_i^1, n_i^2]$ for some $i\in[T-1]$, and $\bw_k=\tw_k$ otherwise.

\item There exists {\red a uniform} constant $\Upsilon_{\text{upper}}$ such that $\forall k\in[N]$:
\begin{equation*}
\|H_k(\bw_k)\| \vee \|A_k(\bz_k)\|\vee \|B_k(\bz_k)\|\leq \Upsilon_{\text{upper}},
\end{equation*}
for any $\bw_k\in\N_\varepsilon(\tw_k)$.	
	
\item There exist {\red uniform} constants $\gamma_C$, $t>0$ such that $\forall k\in[N-t]$ and some $t_k\in[1, t]$:
\begin{equation*}
\Xi_{k, t_k}(\bz_{k:k+t_k-1})\Xi_{k, t_k}(\bz_{k:k+t_k-1})^T\succeq  \gamma_C I,
\end{equation*}
for any $\bz_{k:k+t_k-1}\in\N_\varepsilon(\tz_{k:k+t_k-1})$.
\end{enumerate}
\end{assum}

We now show that the subproblem's {\red solution} recover the {\red truncated} full-horizon solution if perfect boundary conditions are given. We first state the following lemma, which {\red adapts \cite[Lemma 1 and Theorem 1]{Na2020Superconvergence} and suggests that the subproblem formulation \eqref{pro:4}} inherits uniform SOSC from the full problem, provided $\mu$ is sufficiently large.

\begin{lem}\label{lem:mu} 

Let Assumption \ref{ass:4} hold for the local solution~$\bw^\star$ of Problem \eqref{pro:1} and $\varepsilon>0$. There exists $\bar{\mu}$ depending on $(\Upsilon_{\text{upper}}, \gamma_C, t)$ only such that if $\mu\geq \bar{\mu}$, $\bw_{[i]}\in\N_{\varepsilon}(\bw^\star_{[i]})$,~and $\bbw_{[-i]}\in\N_{\varepsilon}(\bw^\star_{[-i]})$, then the reduced Hessian of subproblem $\mP_i(\bbw_{[-i]})$ at $\bw_{[i]}$ is lower bounded by $\gamma_H$ as well. That is,
\begin{equation*}
ReH^i(\bw_{[i]};\bbw_{[-i]})\succeq \gamma_H I,
\end{equation*}
where $ReH^i(\bw_{[i]};\bbw_{[-i]})$ denotes the reduced Hessian of $\mP_i(\bbw_{[-i]})$ evaluated at $\bw_{[i]}$,  defined similarly as in Definition~\ref{def:1}.

\end{lem}

{\red More precisely, we see from \cite[(15)]{Na2020Superconvergence} that
\begin{equation}\label{equ:mu}
\bar{\mu} = \bar{\mu}(\Upsilon_{\text{upper}}, \gamma_C, t) \coloneqq \frac{16\Upsilon_{\text{upper}}(\Upsilon_{\text{upper}}^{6t} - \Upsilon_{\text{upper}}^{4t})}{\gamma_C^2}.
\end{equation}
However, the above expression is only a conservative bound for $\mu$. In our experiments, we will see that $\mu = 1$ works well for different nonlinear OCPs.}

\begin{lem}\label{lem:2.5}
	
Let Assumption \ref{ass:4} hold for the local solution $\bw^\star$ of Problem \eqref{pro:1} and $\varepsilon>0$. We choose $\mu\geq \bar{\mu}$ defined in Lemma \ref{lem:mu}. Then, for any $i\in[T-1]$, $\bw^\star_{[i]}$ is a local solution of $\P_i(\tw_{[-i]})$.

\end{lem}

\begin{proof}

By Lemma \ref{lem:mu}, we know that the reduced Hessian of $\mP_i(\tw_{[-i]})$ evaluated at $\bw^\star_{[i]}$ is lower bounded by $\gamma_H$. Thus, it suffices to show that $\bw^\star_{[i]}$ satisfies the KKT conditions for $\P_i(\tw_{[-i]})$. First, we write the KKT systems for Problem \eqref{pro:1}:
\begin{equation}\label{eqn:kkt-full}
\0 = \left\{\begin{alignedat}{3}
& \nabla_{\bx_k}g_k(\bz_k) + \blambda_{k-1} - A_k^T(\bz_k)\blambda_k, && \forall k\in[N-1],\\
& \nabla_{\bu_k}g_k(\bz_k)  - B_k^T(\bz_k)\blambda_k,  \hskip1.5cm  && \forall k\in[N-1],\\
&\nabla_{\bx_{N}}g_{N}(\bx_{N}) + \blambda_{N-1}, \\
& \bx_{k+1} - f_k(\bz_k), && \forall k\in[N-1],\\
& \bx_{0} - \bbx_{0}.
\end{alignedat}\right.
\end{equation}
Analogously, the KKT system of $\P_i(\bar{\bw}_{[-i]})$ is
\begin{equation}\label{eqn:kkt-sub}
\0 = \left\{\begin{alignedat}{3}
& \nabla_{\bx_k}g_k(\bz_k) + \blambda_{k-1} - A_k^T(\bz_k)\blambda_k, && \forall k\in[n_i^1,n_i^2),\\
& \nabla_{\bu_k}g_k(\bz_k)  - B_k^T(\bz_k)\blambda_k,  \hskip1.5cm  && \forall k\in[n_i^1,n_i^2),\\
&\nabla_{\bx_{n_i^2}}\tilde{g}_{n_i^2}(\bx_{n_i^2}; \bbw_{n_i^2}) + \blambda_{n_i^2-1}, \\
& \bx_{k+1} - f_k(\bz_k), && \forall k\in[n_i^1,n_i^2),\\
& \bx_{n_i^1} - \bbx_{n_i^1},
\end{alignedat}\right.
\end{equation}
where $\nabla_{\bx_{n_i^2}}\tilde{g}_{n_i^2}(\bx_{n_i^2}; \bbw_{n_i^2})$ is
\begin{align*}
&=
\begin{cases}
\nabla_{\bx_{n_i^2}}{g}_{n_i^2}(\bx_{n_i^2}, \bbu_{n_i^2})  - A_{n_i^2}^T(\bx_{n_i^2}, \bbu_{n_i^2})\bblambda_{n_i^2} \\ \quad\quad \quad+ \mu(\bx_{n_i^2} - \bbx_{n_i^2}),\;\;i\in[T-2],\\
\nabla_{\bx_{N}}{g}_{N}(\bx_{N}),\hskip1.6cm  i=T-1.
\end{cases}
\end{align*}
One can see from the satisfaction of KKT system \eqref{eqn:kkt-full} for the full problem \eqref{pro:1} with $\tw$ that \eqref{eqn:kkt-sub} is satisfied with $\tw_{[i]}$ when $\bar{\bw}_{[-i]}=\tw_{[-i]}$. This completes the proof.
\end{proof}

We now estimate errors for the short-horizon solutions. The next theorem characterizes the existence and uniqueness of the local mapping $\bw^\dag_{[i]}(\cdot)$ from the boundary variable $\bw_{[-i]}$ to the local solution of $\P_i(\bw_{[-i]})$.

\begin{thm}\label{thm:sol-conti} 
	
Let Assumption \ref{ass:4} hold for the local solution $\bw^\star$ of Problem \eqref{pro:1} and $\varepsilon>0$. We choose $\mu\geq \bar{\mu}$ defined in Lemma~\ref{lem:mu}. Then, for any $i\in[T-1]$, there exist $\delta>0$, $\varepsilon'\in(0,\varepsilon)$ and a continuously differentiable function $\bw^\dag_{[i]}:\N_{\delta}(\tw_{[-i]})\rightarrow \N_{\varepsilon'}(\tw_{[i]})$ such that, if boundary variable $\bbw_{[-i]} \in \N_{\delta}(\tw_{[-i]})$, then $\bw^\dag_{[i]}(\bbw_{[-i]})$ is a unique local solution of $\mP_i(\bbw_{[-i]})$ in the neighborhood $\N_{\varepsilon'}(\tw_{[i]})$.

\end{thm}

Theorem \ref{thm:sol-conti} is a specialization of the classical result of \cite[Theorem 2.1]{Robinson1974Perturbed}. Since $\bw_{[i]}^\dag$ is differentiable, we analogize the directional derivatives definitions in \eqref{equ:direc:deriva} and, for any point ${\bbw}_{[-i]}$ and perturbation direction $\bl$, define
\begin{equation*}
\bxi^\dag_{[i]}({\bbw}_{[-i]}, \bl) 
=\lim\limits_{h\searrow 0} \frac{\bw^\dag_{[i]}(\bar{\bw}_{[-i]}+h \bl + o(h)) - \bw^\dag_{[i]}(\bar{\bw}_{[-i]})}{h}
\end{equation*}
as the directional derivatives of $\bw^\dag_{[i]}(\bar{\bw}_{[-i]})$. Here, we disregard the perturbation for stages $[n^1_i,n^2_i)$ since in the formulation of subproblems \eqref{pro:4}, only stages $n^1_i-1$ and $n^2_i$ are perturbed. Implementing the exact computation of $\bw^\dag_{[i]}(\bbw_{[-i]})$ is challenging. In practice, one resorts to solving $\mP_i(\bbw_{[-i]})$ using a generic NLP solver, and the solver may return a local solution outside of the neighborhood $\N_{\varepsilon'}(\tw_{[i]})$. Strictly preventing~this is difficult in general. Fortunately, by warm-starting the solver with the previous iterate, one may reduce the chance that the solver returns a solution that is far from the previous iterate. Thus, in our numerical implementation, we implement Algorithm \ref{alg:decom:procedure} by using the warm-start strategy.

The next result characterizes the stagewise difference between $\bw_{[i]}^\dag(\bar{\bw}_{[-i]})$ and $\bw_{[i]}^\dag(\bw^\star_{[-i]})$. 

\begin{thm}\label{thm:eds}

Let Assumption \ref{ass:4} hold for the local solution $\bw^\star$ of Problem \eqref{pro:1} and $\varepsilon>0$. We choose $\mu\geq \bar{\mu}$ defined in Lemma~\ref{lem:mu}, and $\delta>0$ defined in Theorem \ref{thm:sol-conti}. For $i\in[T-1]$, if boundary variable $\bar{\bw}_{[-i]} \in \N_{\delta}(\tw_{[-i]})$, then there exist constants $\Upsilon>0$ and $\rho\in(0,1)$ independent from $N$ and $i$, such that 
\begin{align}\label{eqn:thm:eds}
&\|T_{k}\{\bw^\dag_{[i]}(\bar{\bw}_{[-i]}) - \tw_{[i]}\}\|_{\bw}
\leq  \Upsilon\big\{\rho^{k-n^1_i}\|\bar{\bx}_{n^1_i} - \tx_{n^1_i}\| \nonumber\\ 
& \quad +\rho^{n^2_i-k}\|\bar{\bw}_{n^2_i} - \tw_{n^2_i}\|_{\bw} \big\}, \;\quad\quad  \forall k \in [n^1_i-1,n^2_i].
\end{align}
\end{thm}

\begin{proof}

We define $\Delta\bx_{n^1_i}=\tx_{n^1_i}-\bbx_{n^1_i}$, $\Delta\bw_{n^2_i} = \tw_{n^2_i} - \bbw_{n^2_i}$, and an intermediate boundary variable $\tilde{\bw}_{[-i]} = (\tx_{n^1_i}, \bbw_{n^2_i})$. Then $\tilde{\bw}_{[-i]} \in \N_{\delta}(\tw_{[-i]})$ and thus $\bw^\dag_{[i]}(\tilde{\bw}_{[-i]})$ exist. We have
\begin{align}\label{equ:1}
&\bw^\dag_{[i]}(\bar{\bw}_{[-i]}) - \bw^\dag_{[i]}(\tw_{[-i]})\\
& = \{\bw^\dag_{[i]}(\bar{\bw}_{[-i]}) - \bw^\dag_{[i]}(\tilde{\bw}_{[-i]})\} + \{\bw^\dag_{[i]}(\tilde{\bw}_{[-i]}) - \bw^\dag_{[i]}(\tw_{[-i]})\}.\nonumber
\end{align}
The first term corresponds to the perturbation of the initial stage, while the second term corresponds to the perturbation of the terminal stage. Let us define two directions $\bl_1$ and $\bl_2$~as:
\vskip-0.5cm
\begin{align*}
\bl_1= &
\begin{cases}
\0 & \text{ if }\|\Delta\bx_{n^1_i}\| = 0,\\
\frac{\Delta\bx_{n^1_i}}{\|\Delta\bx_{n^1_i}\|} & \text{otherwise},
\end{cases}\\
\bl_2= &
\begin{cases}
\0 & \text{ if }\|\Delta\bw_{n^2_i}\| = 0,\\
\frac{\Delta\bw_{n^2_i}}{\|\Delta\bw_{n^2_i}\|} & \text{otherwise},
\end{cases}
\end{align*}
and accordingly the perturbation paths $P_1= \{\bbw_{[-i]} + s\bl_1: s\in[0, \|\Delta\bx_{n^1_i}\|]\}$ and $P_2= \{\tilde{\bw}_{[-i]} + s\bl_2: s\in[0,\|\Delta\bw_{n^2_i}\|]\}$. Along $P_1$ the boundary variable changes from $\bbw_{[-i]}$ to $\tilde{\bw}_{[-i]}$; and along $P_2$ the boundary variable changes from $\tilde{\bw}_{[-i]}$ to $\tw_{[-i]}$. We can easily verify from Definition \ref{def:2} that any points on $P_1\cup P_2$ are in the neighborhood $\N_{\delta}(\bw^\star)$. Thus, by Theorem~\ref{thm:sol-conti}, $\bw^\dag_{[i]}(P_1\cup P_2)$ exists and lies in $\N_{\varepsilon'}(\tw_{[i]})$.

The perturbations along the path $P_1\cup P_2$ will be analyzed~using Theorems \ref{thm:primal:sensitivity} and \ref{thm:dual:sensitivity}. We first check that Assumptions \ref{ass:1}, \ref{ass:2}, \ref{ass:3} hold at $\bw^\dag_{[i]}(\bw_{[-i]})$ over $\bw_{[-i]}\in P_1\cup P_2$. Assumption \ref{ass:1} holds at each $\bw^\dag_{[i]}(\bw_{[-i]})$ by Assumption \ref{ass:4}(a) and Lemma \ref{lem:mu}. Assumption \ref{ass:2} holds at each $\bw^\dag_{[i]}(\bw_{[-i]})$ by Assumption \ref{ass:4}(c). For Assumption \ref{ass:3}, we know $H_k$, $A_k$, $B_k$ for $k\in[n_i^1, n_i^2)$ are upper bounded by Assumption \ref{ass:4}(b); further, one can verify that $H_{n^2_i}$, $C_{n^1_i}$, and $D_{n^2_i}$ are also uniformly bounded by inspecting $\mP_i(\cdot)$ and noting that $\mu$ is a parameter independent of $N$ and $i$. Thus, Assumptions \ref{ass:1}, \ref{ass:2}, \ref{ass:3} hold at each $\bw^\dag_{[i]}(\bw_{[-i]})$ over $\bw_{[-i]}\in P_1\cup P_2$.

By Lemma \ref{lem:2.5}, for any $i\in[T-1]$ and $k\in[n^1_i,n^2_{i}]$, we have
\begin{align}\label{eqn:breakdown}
\|T_{k}\{\bw_{[i]}^\dag&(\bar{\bw}_{[-i]})  -  \bw_{[i]}^\star\}\|_{\bw} \nonumber\\
& =\|T_{k}\{\bw_{[i]}^\dag(\bar{\bw}_{[-i]}) -\bw_{[i]}^\dag(\tw_{[-i]})\}\|_{\bw} \nonumber \\
&\stackrel{\eqref{equ:1}}{\leq} \|T_{k}\{\bw_{[i]}^\dag(\bar{\bw}_{[-i]}) - \bw_{[i]}^\dag(\tilde{\bw}_{[-i]}) \}\|_{\bw} \nonumber\\
&\quad + \| T_{k}\{ \bw_{[i]}^\dag(\tilde{\bw}_{[-i]}) - \bw_{[i]}^\dag(\tw_{[-i]}) \}\|_{\bw}.
\end{align}
Rewriting the first term of the right-hand side of the inequality in \eqref{eqn:breakdown} and using the integral of line derivative yields
\begin{align*}
\|T_{k}&\{\bw_{[i]}^\dag(\bar{\bw}_{[-i]}) - \bw_{[i]}^\dag(\tilde{\bw}_{[-i]}) \}\|_{\bw}\\
&= \left\|\int_{0}^{\|\Delta\bx_{n^1_i}\|} T_{k}\{\bxi_{[i]}^\dag(\bbw_{[-i]} + s\bl_1;\bl_1) \}ds \right\|_{\bw}\\
&\leq\int_{0}^{\|\Delta\bx_{n^1_i}\|}\| T_{k}\{\bxi_{[i]}^\dag(\bbw_{[-i]} + s\bl_1;\bl_1) \}\Vert_{\bw} ds\\
&\leq \|\Delta\bx_{n^1_i}\|\cdot \Upsilon\rho^{k - n_i^1} ,
\end{align*}
where the second inequality follows from triangle inequality of integrals and the third inequality follows from Theorems \ref{thm:primal:sensitivity} and \ref{thm:dual:sensitivity}. Similarly, the second term in \eqref{eqn:breakdown} can be bounded by $\Upsilon\rho^{n^2_i-k}\Vert \Delta\bw_{n^2_i}\Vert\leq \sqrt{3}\Upsilon\rho^{n^2_i-k}\Vert \Delta\bw_{n^2_i}\Vert_{\bw}$ {\red where the inequality is due to
\begin{multline*}
\|\Delta\bw_{n^2_i}\| \leq \sqrt{3}(\|\Delta\bx_{n_i^2}\|\vee\|\Delta\bu_{n_i^2}\|\vee \|\Delta\blambda_{n_i^2}\|)\\ = \sqrt{3}\Vert \Delta\bw_{n^2_i}\Vert_{\bw}.
\end{multline*}}
\hskip-3pt Thus, combining these results with \eqref{eqn:breakdown} and letting $\Upsilon\leftarrow \sqrt{3}\Upsilon$, we obtain \eqref{eqn:thm:eds}. This completes the proof.
\end{proof}

Theorem \ref{thm:eds} provides a proof for the conjecture made in \cite[Property 1]{Shin2019Parallel}. It shows that the effect of the perturbation of the boundary variable $\bw_{[-i]}$ of $\P_i$ on the solution $\bw_{[i]}^\dag(\bw_{[-i]})$ decays exponentially as moving away from two boundary ends. Here, the unperturbed data is $\tw_{[-i]}$, and by Lemma \ref{lem:2.5}, the truncated full-horizon solution is the unperturbed subproblem solution, i.e. $\bw_{[i]}^\dag(\tw_{[-i]}) = \tw_{[i]}$. However, since $\tw_{[-i]}$ is unknown, the algorithm uses the previous iterate to specify the boundary variable, which results in a perturbation of $\tw_{[-i]}$. Suppose we perturb $\tw_{[-i]}$ to $\bbw_{[-i]}$, which corresponds to the perturbation of both initial and terminal stages. Theorem \ref{thm:eds} makes use of the EDS property in Theorems \ref{thm:primal:sensitivity} and \ref{thm:dual:sensitivity} and shows that the stagewise error brought by the perturbation decays exponentially. In particular, for $k \in[m_{i-1}, m_{i}]$, we note that $k - n_i^1 \wedge n_i^2-k \geq \tau$, which implies the stagewise error within $[m_{i-1}, m_{i}]$ has been improved by at least a factor $2\Upsilon\rho^\tau$. Thus, if $\tau \geq \log(2\Upsilon)/\log(1/\rho)$, we can observe a clear improvement for middle stages $[m_i, m_{i+1}]$. This justifies discarding iterates for the subdomain overlaps and concatenating iterates of the non-overlapping subdomains.

We are now in a position to establish our main convergence result. Based on Lemma \ref{lem:2.5} and Theorem \ref{thm:eds}, we establish the local convergence of Algorithm \ref{alg:decom:procedure}.

\begin{thm}\label{thm:conv}

Let Assumption \ref{ass:4} hold for the local solution $\bw^\star$ of Problem \eqref{pro:1} and $\varepsilon>0$. There exist parameters $\bar{\mu}$, $\bar{\tau}>0$, and a constant $\delta>0$ such that if $\mu\geq \bar{\mu}$, $\tau\geq \bar{\tau}$, and $\bw^{(0)} \in \N_{\delta}(\tw)$, the following holds for $\ell=0,1,\cdots$:
\begin{align}\label{eqn:conv}
\Vert \bw^{(\ell)}-\bw^\star\Vert_{\bw}\leq \alpha^\ell \Vert \bw^{(0)}-\bw^\star\Vert_{\bw},
\end{align}
where $\alpha = 2\Upsilon \rho^\tau$ is independent of $N$.
\end{thm}

\begin{proof}

We choose $\bar{\mu}$ in Lemma \ref{lem:mu}, $\bar{\tau}=\lceil\log(2\Upsilon)/\log(1/\rho)\rceil+1$, and $\delta$ defined in Theorem \ref{thm:sol-conti}. Then $\alpha< 1$. We first show ${\bw}^{(\ell)}\in\N_{\delta}(\tw)$ for $\ell=0,1,\ldots$ using mathematical induction. The claim trivially holds for $\ell=0$ from the assumption. Assume that the claim holds for $\ell$; thus ${\bw}_{[-i]}^{(\ell)}\in\N_{\delta}(\tw_{[-i]})$. From Theorem \ref{thm:eds} and noting that for any $i\in[T-1]$ and any $\forall k\in[m_i, m_{i+1})$, $k - n_i^1\wedge n_i^2 - k \geq \tau$, we obtain
\begin{equation*}
\|\bw^{(\ell+1)}_k - \tw_k\|_{\bw}
\leq \alpha \|{\bw}^{(\ell)}_{[-i]} - \tw_{[-i]}\|_{\bw}\leq \alpha \|{\bw}^{(\ell)} - \tw\|_{\bw}.
\end{equation*}
Taking maximum over $k$ on the left hand side,
\begin{equation}\label{eqn:final}
\|\bw^{(\ell+1)} - \tw\|_{\bw}
\leq \alpha \|{\bw}^{(\ell)} - \tw\|_{\bw}.
\end{equation} 
From $\alpha<1$, we have $\bw^{(\ell+1)}\in\N_\delta(\tw)$ and the induction is complete. Recursively using \eqref{eqn:final}, we obtain \eqref{eqn:conv}.
\end{proof}

Theorem \ref{thm:conv} establishes local linear  convergence of Algorithm \ref{alg:decom:procedure}. In summary, if SOSC, controllability condition, and boundedness condition are satisfied around the neighborhood of the local primal-dual solution of interest, the algorithm converges to the solution at a linear rate  (provided that the overlap size $\tau$ is sufficiently large). Furthermore, the convergence rate $\alpha$ decays exponentially in $\tau$. One may observe that the overlap size may reach the maximum (i.e., $[n^1_i,n^2_i]=[0,N]$ for $i\in[T-1]$) before $\alpha<1$ is achieved. In that case, the algorithm converges in one iteration (it becomes a centralized algorithm). However, since $\Upsilon$ and $\rho$ are parameters independent of $N$, when a problem with a sufficiently long horizon is considered, one can always obtain the exponential improvement of the convergence rate before reaching the maximum overlap. {\red The argument states the existence of $\bar{\mu}$, $\bar{\tau}$ and $\delta$. The expression of $\bar{\mu}$ is provided in \eqref{equ:mu}; $\bar{\tau}$ is selected by enforcing $\alpha = 2\Upsilon\rho^{\bar\tau}<1$ so that $\bar\tau =\\ \lceil\log(2\Upsilon)/\log(1/\rho)\rceil +1$. Here, $\Upsilon, \rho$ are constants from Theorem \ref{thm:eds}, originating from the sensitivity analysis presented in the end of Section \ref{sec:2}. Their expressions in terms of the constants in Assumption \ref{ass:4} can be found in \cite{Na2020Exponential}, but we do not present them here observing that they are often conservative in practice. As typical for local convergence analysis for NLP, the local radius $\delta$ is not accessible due to the intrinsic nonlinearity of the problem (e.g., see Theorems 11.2 and 18.4 in the textbook \cite{Nocedal2006Numerical} for comparison), while it is also not required for performing our algorithm.
}

%\begin{remark}\label{rem:2}

{\red As stated in Theorem \ref{thm:conv}, Algorithm \ref{alg:decom:procedure} has two tuning~parameters $\mu$ and $\tau$,} where the former is the penalty parameter for the subproblem (cf. \eqref{pro:4}) and the latter is the overlap size. We note that the quadratic penalty is only added to the terminal state variable. Thus a very large $\mu$ 
%%does not lead to ill-conditioning issue for subproblems, but 
is equivalent to fixing the terminal state $\bx_{n^2_{i}}$ at $ \bbx_{n_i^2}$. We can let $\mu$ be large to ensure the condition $\mu\geq \bar{\mu}$ to be satisfied, though our experiments show that a moderate $\mu$ also works well in practice. On the other hand, a larger $\tau$ implies faster convergence, but also results in longer subproblems. In practice, we may tune $\tau$ to balance the fast convergence rate and the increased subproblem complexity. Moreover, one benefit from our analysis is that tuning $\mu$ and $\tau$ is independent from horizon length $N$. Thus, we only target the subproblems with fixed, short horizons to tune $\mu$ and $\tau$, and the same parameters work even when $N$ is extremely large.

%\end{remark} 

The convergence of Algorithm \ref{alg:decom:procedure} can be monitored by checking the KKT residuals of \eqref{pro:1}. However, a more convenient surrogate of the full KKT residuals can be derived as follows. 

\begin{prop}\label{prop:criteria} 

Let $\{\bw^{(\ell)}\}_{\ell=0}^\infty$ be the sequence generated by Algorithm \ref{alg:decom:procedure} with $\tau\geq 1$. Any limit point of the sequence satisfies the KKT conditions \eqref{eqn:kkt-full} for the full problem \eqref{pro:1} if the following is satisfied for $i\in(0,T)$ as $\ell\rightarrow\infty$:
\begin{align*}    
\begin{cases}
T_{m_i}(\bx^\dag_{[i-1]}(\bw^{(\ell)}_{[-(i-1)]})) - \bx^{(\ell+1)}_{m_i}  \rightarrow 0,\\
T_{m_i-1}(\blambda^\dag_{[i]}(\bw^{(\ell)}_{[-i]})) - \blambda^{(\ell+1)}_{m_i-1}  \rightarrow 0.
\end{cases}
\end{align*}
\end{prop}

\begin{proof}

Recalling that each $\bw^\dag_{[i]}(\bw^{(\ell)}_{[-i]})$ satisfies the KKT conditions of $\mP(\bw^{(\ell)}_{[-i]})$,  one can observe that the KKT conditions \eqref{eqn:kkt-full} for the full problem \eqref{pro:1} are violated only in the first equation of \eqref{eqn:kkt-full} over $k\in\{m_i\}_{i=1}^{T-1}$ and in the fourth equation of \eqref{eqn:kkt-full} over $k\in\{m_i-1\}_{i=1}^{T-1}$; and the residuals at iteration $\ell+1$ are $T_{m_i-1}(\blambda^\dag_{[i]}(\bw^{(\ell)}_{[-i]})) - \blambda^{(\ell+1)}_{m_i-1} $ and $T_{m_{i}}(\bx^\dag_{[i-1]}(\bw^{(\ell)}_{[-(i-1)]})) - \bx^{(\ell+1)}_{m_{i}}$, respectively. Therefore, by the given condition, we have that \eqref{eqn:kkt-full} is satisfied for any limit points of the sequence.
\end{proof}

Accordingly, we define the monitoring metrics by
\begin{align*}
\epsilon_{\text{pr}}^{(\ell)}&=\max_{i\in(0,T)}\Vert T_{m_i}(\bx^\dag_{[i-1]}(\bw^{(\ell)}_{[-(i-1)]})) - \bx^{(\ell+1)}_{m_i}\Vert, \\
\epsilon_{\text{du}}^{(\ell)}&=\max_{i\in(0,T)}\Vert T_{m_i-1}(\blambda^\dag_{[i]}(\bw^{(\ell)}_{[-i]})) - \blambda^{(\ell+1)}_{m_{i}-1}\Vert,
\end{align*}
and then set the convergence criteria by
\begin{equation*}
\text{stop if: }\epsilon_{\text{pr}}^{(\ell)}\leq \epsilon_{\text{pr}}^{\text{tol}}\text{ and }\epsilon_{\text{du}}^{(\ell)}\leq \epsilon_{\text{du}}^{\text{tol}} ,
\end{equation*}
for the given tolerance values $\epsilon_{\text{pr}}^{\text{tol}}$, $\epsilon_{\text{du}}^{\text{tol}}$.

Before showing numerical results, we discuss the global behavior of Schwarz decomposition. In general, there is no guarantee for the scheme to converge globally for nonlinear OCPs. As shown in Theorem \ref{thm:sol-conti}, the solution mapping $\bw_{[i]}^\dag$ exists only in a neighborhood of $\tw_{[-i]}$ and our main rate of convergence result, Theorem \ref{thm:conv}, is strictly local in nature. Outside of the neighborhood, $\P_i(\bw_{[-i]})$ may have infinite solutions or may have no solution due to nonlinearity. In~general, we would need a merit function and a line search to ensure local convergence to a stationary point globally \cite{Nocedal2006Numerical}. However, when we have more structure in the problem, the scheme can converge globally. For example, we show in the next theorem that Algorithm \ref{alg:decom:procedure} converges globally for linear-quadratic OCPs. More general results using merit functions will be investigated in future research.

\begin{thm}\label{thm:LQR}
	
Let us consider linear-quadratic problems with
\begin{align*}
g_k(\bx_k, \bu_k) &=  \begin{pmatrix}
\bx_k\\
\bu_k
\end{pmatrix}^T\underbrace{\begin{pmatrix}
	Q_k & S_k^T\\
	S_k & R_k
	\end{pmatrix}}_{H_k}\begin{pmatrix}
\bx_k\\
\bu_k\\
\end{pmatrix}
  +
  \begin{pmatrix}
    \br_k\\
    \bs_k\\
  \end{pmatrix}^T
  \begin{pmatrix}
    \bx_k\\
    \bu_k\\
  \end{pmatrix}
  \\
g_N(\bx_N) &= \bx_N^T Q_N \bx_N + \br^T_N \bx_N
\end{align*}
and $f_k(\bx_k, \bu_k) = A_k\bx_k + B_k\bu_k + \bv_k$. Suppose Assumption \ref{ass:4} holds for $\{H_k, A_k, B_k\}$. Then there exist parameters $\bar{\mu}, \bar{\tau}>0$ such that if $\mu\geq \bar{\mu}$, $\tau\geq \bar{\tau}$, the linear convergence to the unique global solution $\tw$ of Problem \eqref{pro:1}, shown in \eqref{eqn:conv}, holds for any initial point $\bw^{(0)}$.

\end{thm}

We note that $\{H_k, A_k, B_k\}$ do not depend on $\bw_k$, and linear terms $\{\br_k,\bs_k,\bv_k\}$ do not affect the Lagrangian Hessian and constraint Jacobian. Thus, Assumption \ref{ass:4} holds in the whole space.

\begin{proof}

First, each subproblem \eqref{pro:4} is still a linear quadratic problem with any boundary variable $\bw_{[-i]}$. By Lemma \ref{lem:mu}, there exists $\bar{\mu}$ such that $\mu\geq \bar{\mu}$ implies $ReH^i \succeq \gamma_H I$. Moreover, by LICQ of \eqref{pro:4} (i.e. the Jacobian $G$ has full row rank, see Definition \ref{def:1}), we know from \cite[Lemma 16.1]{Nocedal2006Numerical} that each subproblem has a unique, global solution for any $\bw_{[-i]}$, denoted by $\bw_{[i]}^\dag(\bw_{[-i]})$. Thus, Theorem \ref{thm:sol-conti} is applicable with the stated neighborhoods being the entire space. 

Second, from Lemma \ref{lem:2.5}, we know $\bw_{[i]}^\dag(\tw_{[-i]}) = \tw_{[i]}$.~Moreover, the one-step error recursion in \eqref{eqn:thm:eds} holds directly~following the same proof as in Theorem \ref{thm:eds}, and finally \eqref{eqn:final} holds as well. This shows \eqref{eqn:conv} holds. 
\end{proof}

{\red Although Theorem \ref{thm:LQR} is limited to the linear-quadratic~settings, we allow for the possibility of negative curvature in the objective. Specifically, note that Assumption \ref{ass:4}(a) only requires $ReH\succeq \gamma_H I$, while existing results  \cite{Xu2018Exponentially} require the strong convexity $H\succeq \gamma_H I$.} This benefit comes from the convexification procedure in Algorithm \ref{alg:convex:proce} that our EDS results are based on.

%\begin{remark}
It is worthwhile to mention that computation of the Newton step (search direction) for nonlinear OCPs effectively reduces to solving a linear-quadratic OCP. {\red Accordingly, the overlapping Schwarz scheme can be used as a method to compute the search directions within second-order methods (such as interior-point approaches). Theorem \ref{thm:LQR} provides the global convergence proof for overlapping Schwarz-based step computations.} We acknowledge that there is a wide range of decomposition methods for linear-quadratic OCPs (that include both iterative and direct approaches). Studying and comparing the scalability of these methods with that of the Schwarz scheme is an interesting direction of future work. 
%\end{remark}

\section{Numerical Experiments}\label{sec:6}

We apply the overlapping Schwarz scheme to a quadrotor control problem (governed by equations of motions) and to a thin plate temperature control problem (governed by nonlinear PDEs). Here, we aim to illustrate the convergence behavior of the Schwarz scheme and also to compare performance against state-of-the-art approaches such as ADMM and a centralized interior-point solver (Ipopt). We will demonstrate that Schwarz provides flexibility (as ADMM) and efficiency (as Ipopt).~Our results also illustrate how EDS property arises in applications.

\subsection{Quadrotor Control}
We consider a quadrotor model studied in \cite{Hehn2011flying, Deng2019parallel}:
\begin{align*}
\frac{d^2X}{dt^2} &= a(\cos\gamma\sin\beta\cos\alpha+\sin\gamma\sin\alpha)\\
\frac{d^2Y}{dt^2} &= a(\cos\gamma\sin\beta\sin\alpha-\sin\gamma\cos\alpha)\\
\frac{d^2Z}{dt^2} &= a\cos\gamma\cos\beta-\overline{g}\\
\frac{d\gamma}{dt}&= (\omega_X\cos\gamma + \omega_Y\sin\gamma)/\cos\beta\\
\frac{d\beta}{dt}&= -\omega_X\sin\gamma + \omega_Y\cos\gamma\\
\frac{d\alpha}{dt}&= \omega_X\cos\gamma\tan\beta + \omega_Y\sin\gamma\tan\beta\ + \omega_Z.
\end{align*}
Here, $(X,Y,Z)$ are the positions, $(\gamma,\beta,\alpha)$ are the angles,~and $\overline{g}=9.8$ is the gravitational acceleration. We regard $\bx=(X,\dot{X},Y,\dot{Y},Z,\dot{Z},\gamma,\beta,\alpha)$ as the state variables and $\bu=(a,\omega_X,\omega_Y,\omega_Z)$ as the control variables. The dynamics are discretized with an explicit Euler scheme with time step $\Delta t=0.005$  to obtain an OCP of the form of interest.  Furthermore, the stage cost function is $g_k(\bx_k,\bu_k)=\frac{1}{2}\Vert \bx_k-\bx^{\text{ref}}_k\Vert_Q^2+\frac{1}{2}\Vert \bu_k\Vert_R^2$; $g_N(\bx_N)= \frac{1}{2\Delta t}\Vert \bx_N-\bx^{\text{ref}}_N\Vert_Q^2$; $Q=\diag(1,0,1,0,1,0,1,1,1)$; $R=\diag(1/10,1/10,1/10,1/10)$; $\{\bx^{\text{ref}}_k\}_{k=1}^N$ is generated from a sinusoidal function; $N=24,000$ (full problem; corresponds to 60 secs horizon); $\mu=1$; and $\bar{\bx}_0=(0;0;0;0;0;0;0;0;0)$.

\subsection{Thin Plate Temperature Control}
We consider a thin plate temperature control problem \cite{mathworks2020nonlinear}:
\begin{subequations}\label{pro:5}
\begin{align}
\min_{x,u}\;\;& \int_0^T\int_{w\in \Omega}\frac{1}{2}(x(w,t)-d(w,t))^2 + \frac{1}{2}ru(w,t)^2\; dwdt\\
\st\;\; & \frac{\partial x(w,t)}{\partial t} =-\Delta x(w,t) + \frac{2h_c}{\kappa t_z} (x(w,t)-\overline{T}) \nonumber\\
&\quad\quad + \frac{2\epsilon\sigma}{\kappa t_z} (x(w,t)^4-\overline{T}^4)- \frac{1}{\kappa t_z}u(w,t),\nonumber\\
&\quad\quad w\in \Omega,t\in[0,T] \label{pro:5b}\\
&x(w,t) = \overline{T},\; \quad w\in \partial\Omega,
\end{align}
\end{subequations}
where $\Omega=[0,1]\times[0,1]\subseteq\mathbb{R}^2$ is the 2-dimensional domain of interest; $x:\Omega\times [0,T]\rightarrow \mathbb{R}$ is the temperature; $u:\Omega\times [0,T]\rightarrow \mathbb{R}$ is the control; $\Delta$ is the Laplacian operator; $\partial\Omega$ is the boundary of $\Omega$; $d:\Omega\times [0,T]\rightarrow\mathbb{R}$ is the desired temperature; $r=0.1$, $\kappa=400$, $t_z=0.01$, $h_c=1$, $\epsilon=0.5$, $\sigma=5.67\times 10^{-8}$, and $\overline{T}=300$ are the problem parameters (see \cite{mathworks2020nonlinear}). The desired temperature data are generated from a sinusoidal function. The PDE in \eqref{pro:5b} is governed by the heat equation which consists of convection, radiation, and forcing terms, and the Dirichlet boundary condition is enforced. We consider a discretized version of the problem: we discretize $\Omega$ by a $10\times 10$ mesh and $T=24$ hour prediction horizon with $\Delta t = 10$ secs.

\subsection{Methods and Results}

We first present a numerical verification of primal-dual EDS (Theorem \ref{thm:eds}) using the quadrotor problem. We first obtain the reference primal-dual solution trajectory by solving the full problem. Then, the perturbed trajectories are obtained by solving the problem with the perturbation on the given initial state $\bbx_{0}$, and on the terminal state $\bx^{\text{ref}}_N$. In particular, we solved the full problem with random perturbations $\Delta\bar{\bx}_0$ and $ \Delta \bx^\text{ref}_N$ drawn from a zero-mean normal distribution. 

The reference trajectory and 30 samples of the perturbed trajectories are shown in Fig. \ref{fig:eds}. One can see that the solution trajectories coalesce in the middle of the time domain and increase the spread at two boundaries. This result indicates that the sensitivity is decreasing toward the middle of the interval and verifies our theoretical results. 

\begin{figure}
  \centering
  \includegraphics[width=.41\textwidth]{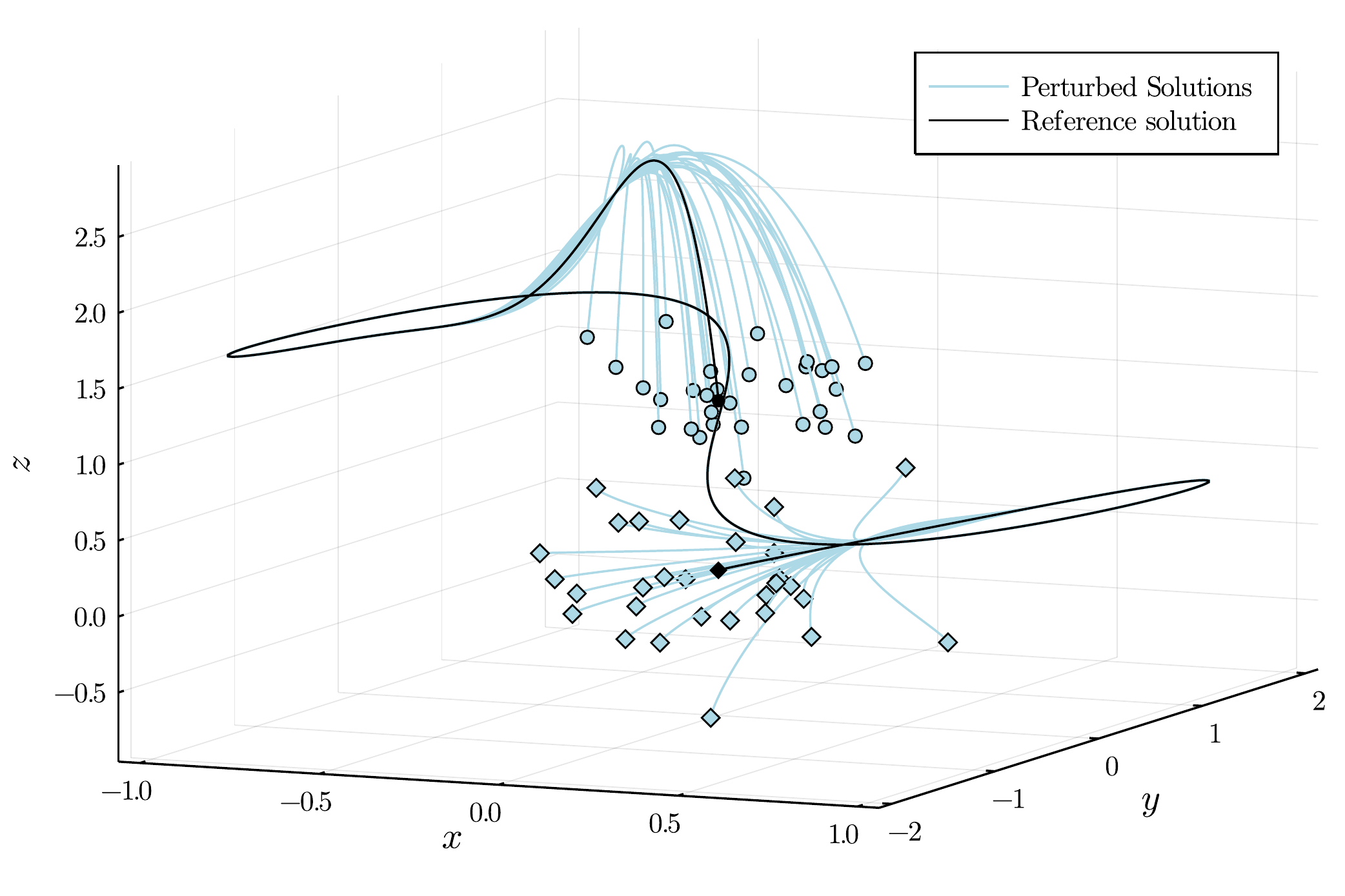}\\
  \includegraphics[width=.41\textwidth]{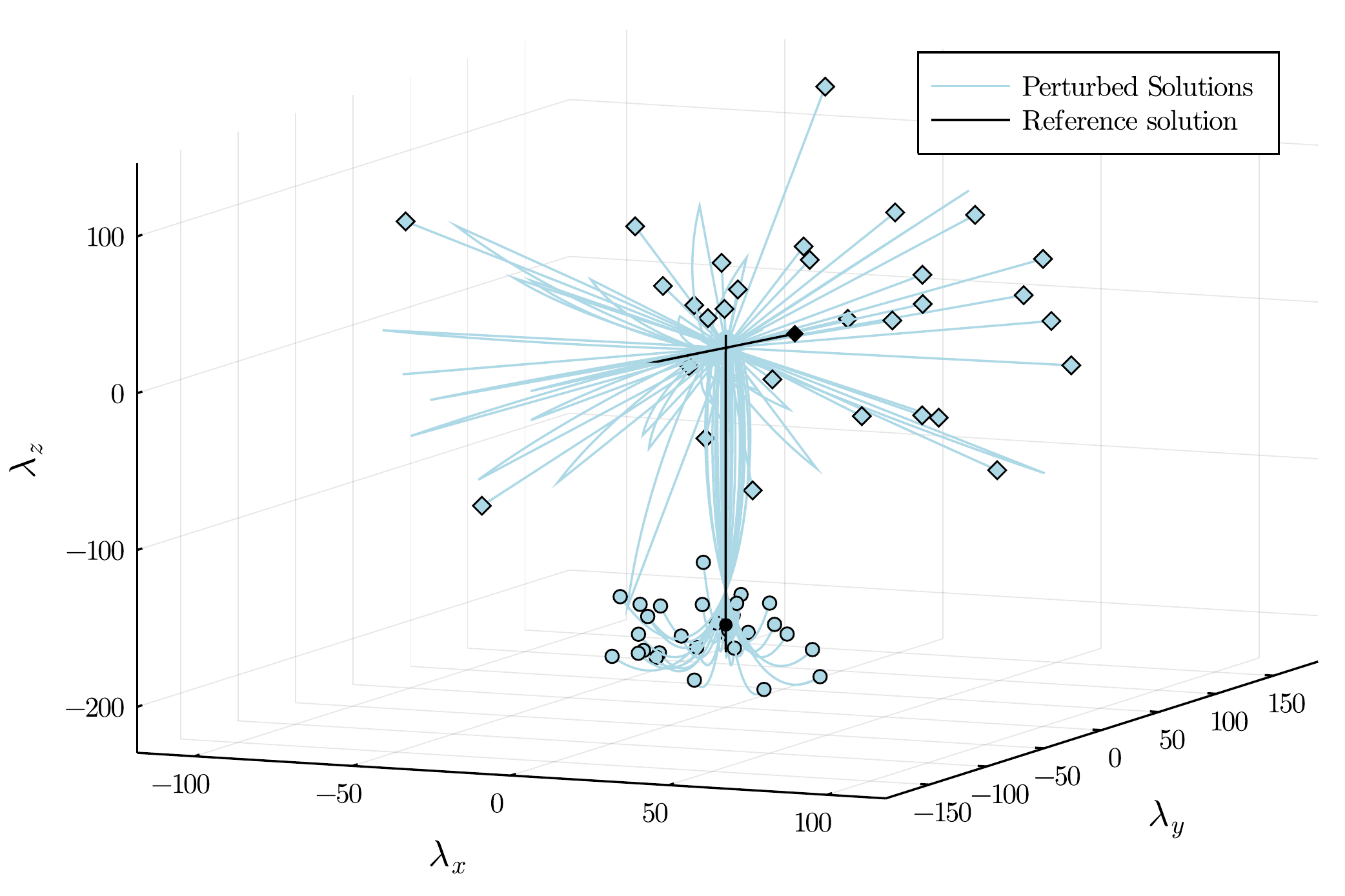}
  \caption{Primal-dual exponential decay of sensitivity for quadrotor problem. The black line represents the reference trajectory; the light blue lines represent the perturbed trajectories; the diamonds represent the initial state; and the circles represent the terminal state. Observe the collapse of the perturbed trajectories to the reference one in the middle of the time intervals.}\label{fig:eds}
\end{figure}

\begin{table}[hb]\centering
  \caption{Iterations and solution time as a function of overlap size.}\label{tbl:conv}
  \begin{tabular}{|c|c|c|c|}
    \hline
    &$\widetilde{\tau}=0.3$&$\widetilde{\tau}=0.5$&$\widetilde{\tau}=1.0$\\
    \hline
    Iterations &31&11&7\\
    \hline
    Solution Time (sec)&6.03 & 2.41 & 2.46\\
    \hline
  \end{tabular}
\end{table}

We now illustrate the convergence of the Schwarz scheme for the quadrotor problem with 3 subdomains. The evolution of KKT errors with different overlap sizes are plotted in Fig. \ref{fig:conv}. Here, we expand the domain until the size of the extended domain reaches $\widetilde{\tau}$ times the original non-overlapping domain. Such relative criteria are often more practical because the scaling of the problem changes with discretization mesh size. We observe that the convergence rate improves dramatically as $\widetilde{\tau}$ increases. This result verifies Theorem \ref{thm:conv}. Fig. \ref{fig:conv-2} further illustrates convergence of the trajectories for $\widetilde{\tau}=0.1$. At the first iteration, the error is large at the boundaries and small in the middle of the domain. The error decays rapidly as the high-error components of the solution are discarded and the low-error components are kept. This behavior illustrates why EDS is central to achieve convergence. A computational trade-off exists for the Schwarz scheme when increasing $\widetilde{\tau}$ (since the subproblem complexity increases with $\widetilde{\tau}$). This trade-off is revealed from time per iteration (see Table \ref{tbl:conv}): we find the scheme takes 0.219 sec/iter when $\widetilde{\tau}=0.5$ and 0.352 sec/iter when $\widetilde{\tau}=1.0$.

\begin{figure}
\centering
  \includegraphics[width=.41\textwidth]{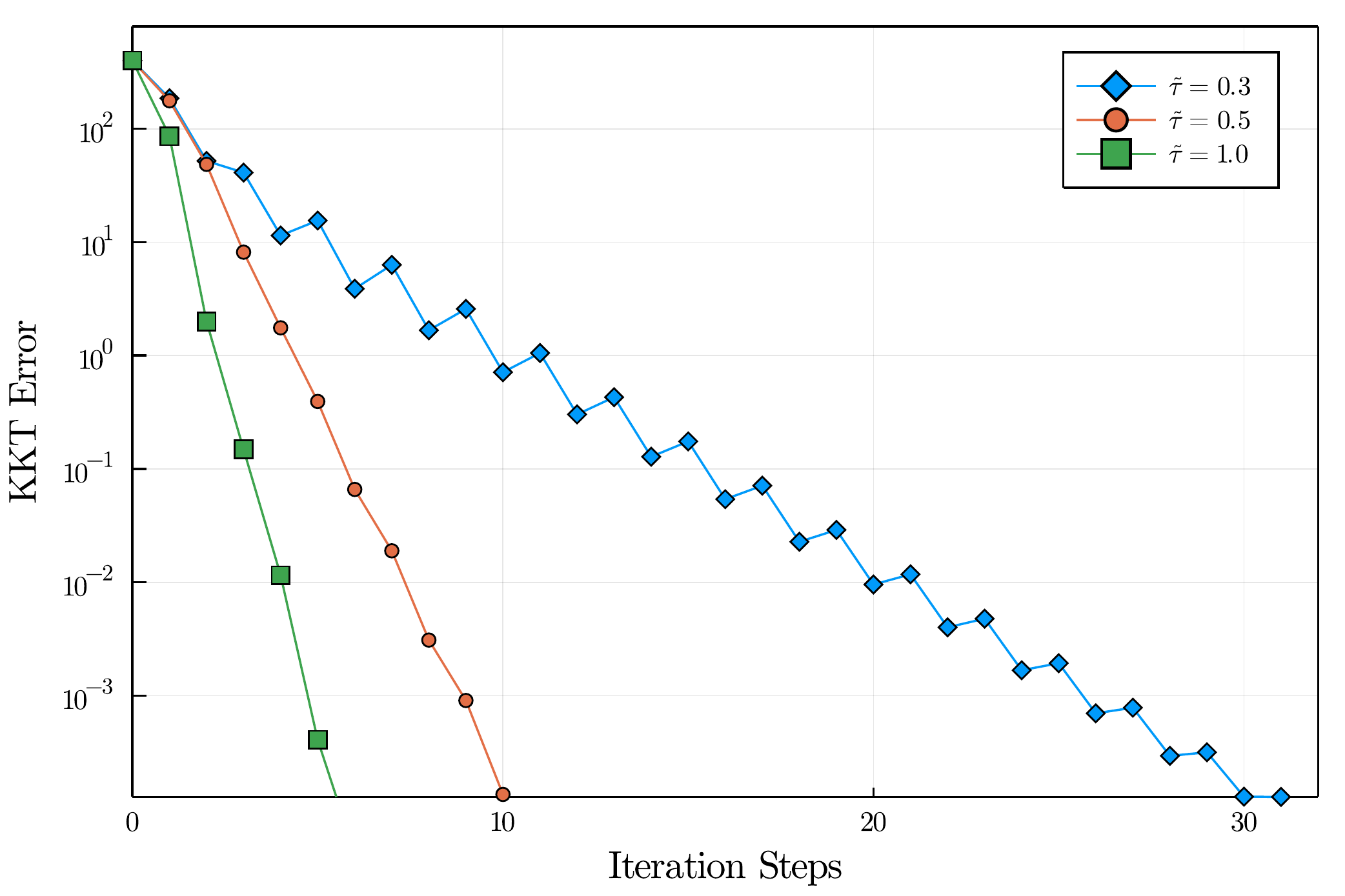}
  \caption{Convergence of KKT residuals for overlapping Schwarz scheme.}\label{fig:conv}
\end{figure}
\begin{figure}
  \centering
    \includegraphics[width=.41\textwidth]{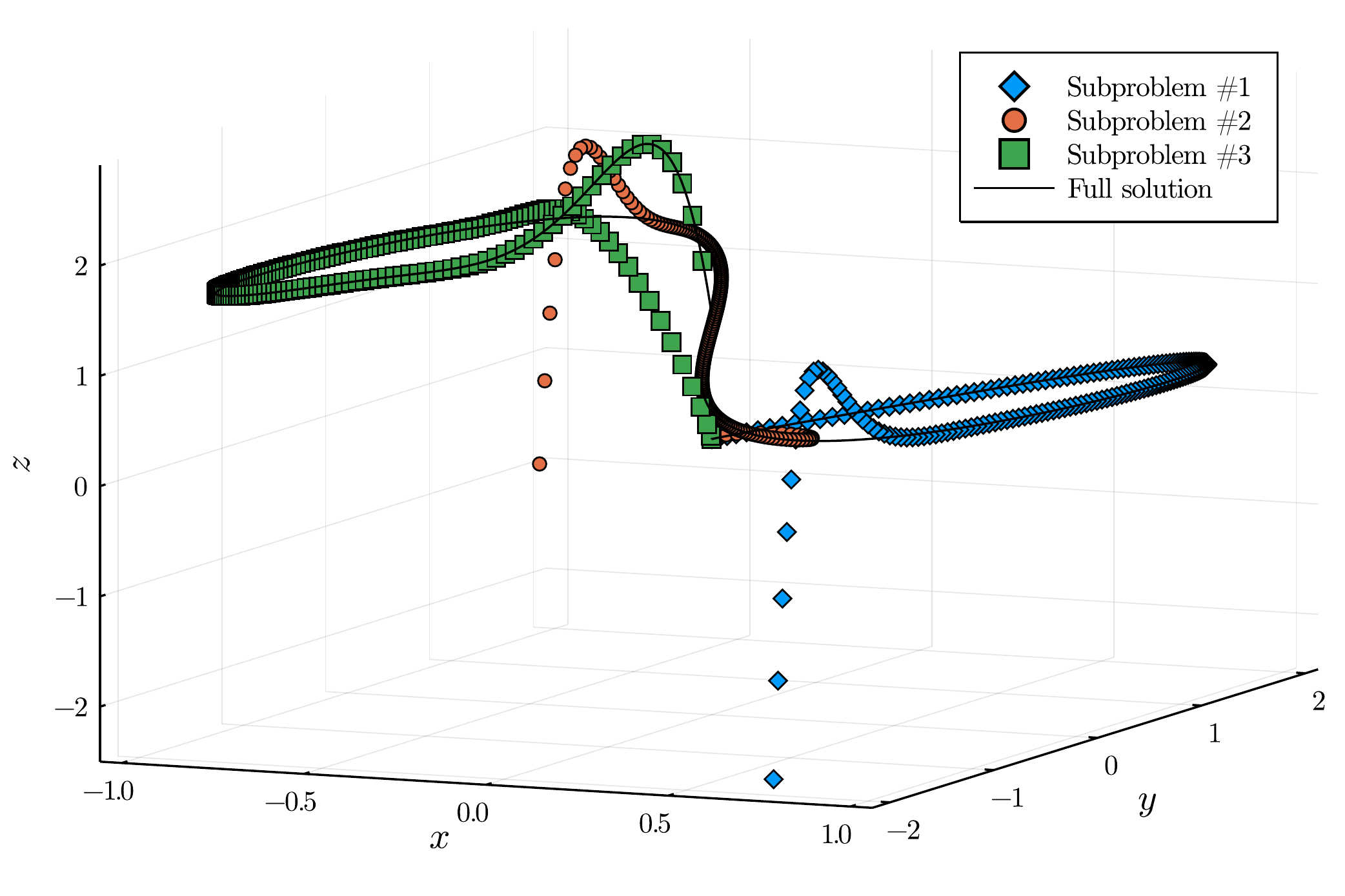}\\
    \includegraphics[width=.41\textwidth]{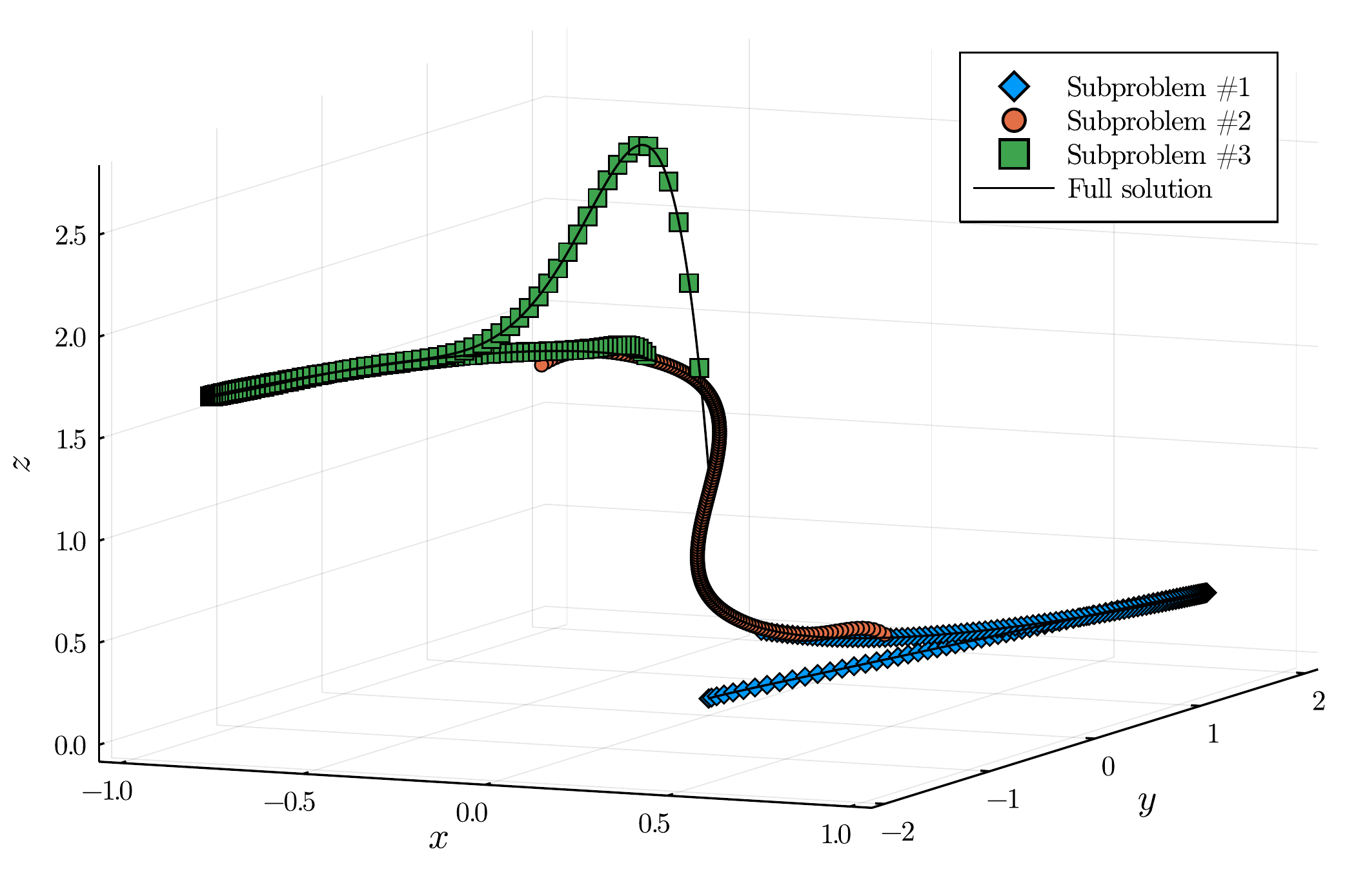}\\
    \includegraphics[width=.41\textwidth]{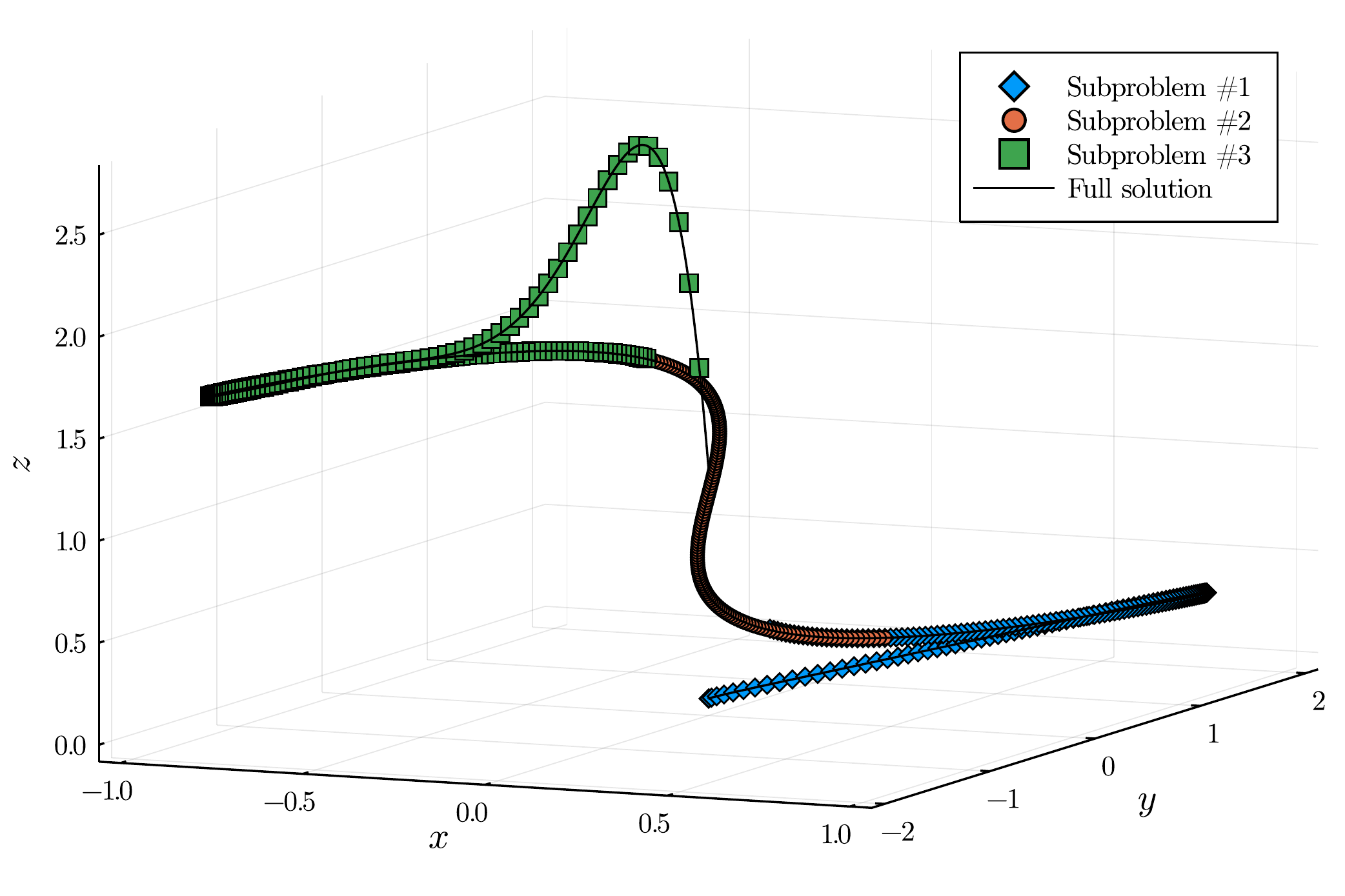}
  \caption{Convergence of primal trajectory with $\tilde{\tau} = 0.1$. Top-to-bottom: iterations 1,2, and 3; blue, red, green markers are solutions from subproblems 1, 2, and 3, respectively; black line is solution trajectory.}\label{fig:conv-2}
\end{figure}

We also benchmark the Schwarz scheme against a centralized solver (Ipopt) and against a popular decomposition scheme (ADMM) for solving the above two problems. {\red For both the Schwarz and ADMM schemes}, we partition the domain into 20 intervals with the same length. {\red For the Schwarz scheme, we expand each interval by \eqref{equ:n1in2i} with the relative size of overlap $\tilde{\tau} = 1.0$.} For ADMM, subproblems are formulated by introducing duplicate variables and decomposing on the time domain \cite{Rodriguez2018Benchmarking}. To ensure consistency, subproblems in the Schwarz and ADMM schemes are all solved with Ipopt \cite{Waechter2005implementation}, configured with the sparse solver {\tt MA27} \cite{A2007collection}. The study is run on a multicore parallel computing server (shared memory and 40 cores of Intel Xeon Gold 6140 CPU running at 2.30GHz) using the multi-thread parallelism in {\tt Julia}. {\red For both the Schwarz and ADMM schemes, we vary the penalty parameter as indicated in Figure \ref{fig:benchmark}.} All results can be reproduced using the provided scripts in \url{https://github.com/zavalab/JuliaBox/tree/master/SchwarzOCP}.

One can see that, for both problems, the overlapping Schwarz scheme has much faster convergence than ADMM (Fig. \ref{fig:benchmark}) {\red regardless of the choice of $\mu$. One can also observe that the performance of ADMM is sensitive to the choice of $\mu$, while the performance Schwarz scheme is insensitive to it.} We note that ADMM tends to decrease the overall error but eventually the error settles to a rather high value. We also emphasize that ADMM does not have convergence guarantees for the general nonconvex OCPs considered here {\red (as discussed in Section \ref{sec:1})}. In contrast, the overlapping Schwarz scheme converges almost as fast as centralized solver Ipopt. The final accuracy of the Schwarz scheme (10${}^{-6}$) is much higher than that achieved by ADMM, but not as that achieved by Ipopt (less than 10${}^{-8}$). The difference between the accuracy between Schwarz and Ipopt are due to the fact that Schwarz is an iterative scheme, while the linear algebra performed inside Ipopt uses a direct linear solver (MA27). Direct solvers are known for delivering high accuracy. We highlight, however, that in {\red many} control applications there is often flexibility to deliver moderate accuracy, {\red so being able to deliver moderately accurate solution in a reasonably short time can be a favorable characteristic}. For the thin plate temperature control problem, the accuracy of ADMM is notably worse than that achieved with the Schwarz scheme. 

To sum up, the overlapping Schwarz scheme is an efficient method to solve OCPs and offers flexibility to be implemented in different computing hardware architectures. 

\begin{figure}
\centering
\includegraphics[width=.41\textwidth]{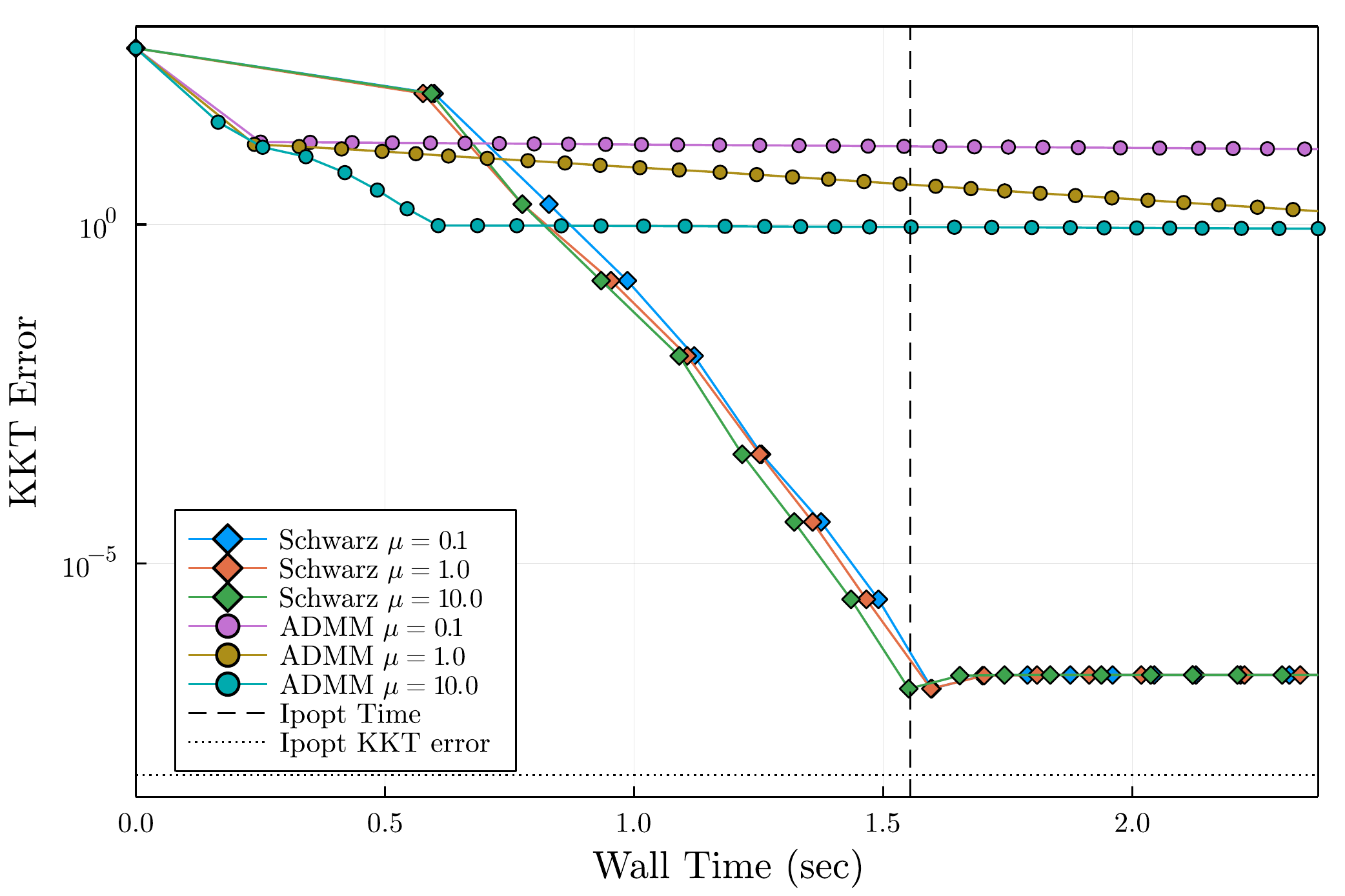}\\
\includegraphics[width=.41\textwidth]{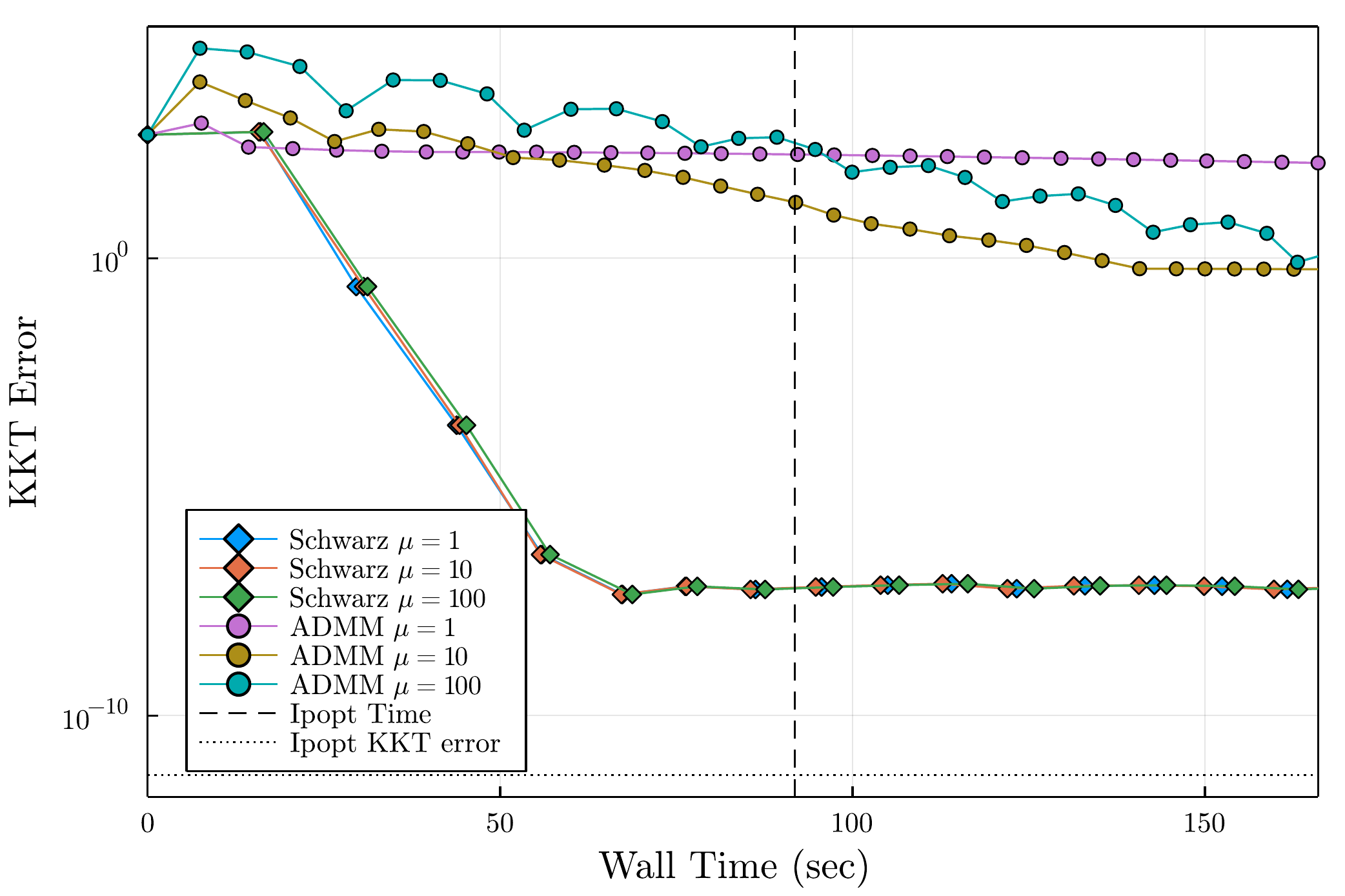}
\caption{Benchmark of overlapping Schwarz against Ipopt and ADMM. Top: Quadrotor control problem; bottom: Thin plate temperature control problem.}\label{fig:benchmark}
\end{figure}

\section{Conclusions}\label{sec:7}

We established the convergence properties of an overlapping Schwarz decomposition scheme for general nonlinear optimal control problems. Under standard SOSC, controllability, and boundedness conditions for the full problem, we showed~that the scheme enjoys linear convergence locally, with a linear~rate that improves exponentially with the size of the overlap. Central to our convergence proof is a primal-dual parametric sensitivity result that we call exponential decay of sensitivity. We also provided a global convergence proof for the Schwarz scheme for the linear-quadratic OCP case. This result is of relevance, as it suggests that the scheme could be used to solve linear systems inside NLP solvers. Computational results reveal that the Schwarz scheme is significantly more efficient than ADMM and as efficient as the centralized NLP solver Ipopt. In future work, we will seek to expand our results to alternative problem structures (e.g., networks and stochastic programs). Moreover, it will be interesting to compare performance in different hardware architectures (e.g., embedded systems) and against different decomposition schemes (e.g., Riccati and block cyclic reduction).

\section*{Acknowledgment}
This material is based upon work supported by the U.S. Department of Energy, Office of Science, Office of Advanced Scientific Computing Research (ASCR) under Contract DE-AC02-06CH11347 and by NSF through award CNS-1545046 and ECCS-1609183. We also acknowledge partial support from the National Science Foundation under award NSF-EECS-1609183.

\section*{Appendix}

The missing proofs in Section \ref{sec:3} are presented here.

\subsection{Proof of Theorem \ref{thm:1}}\label{pf:thm:1}

Under Assumption \ref{ass:1}, we know that $(\tp, \tq, \tzeta)$ is a unique global solution of $\mL\mQ\mP$. When executing Algorithm~\ref{alg:convex:proce} with $\beta\in(0, \gamma_H)$, {\red we know from \cite[Theorem 3.8]{Na2020Exponential} that $\tilde{H}_k\succ 0$ (i.e. $\tilde{H}_k(\beta)$ in their notation).} Thus, $(\tpc, \tqc, \tzetac)$ is also a~unique global solution of $\mC\mL\mQ\mP$. By \cite[Lemma 3.4]{Na2020Exponential}, {\red we know that $\mL\mQ\mP$ and $\mC\mL\mQ\mP$ have the same objective. Since Algorithm \ref{alg:convex:proce} does not change constraint matrices $A_k, B_k, C_k$ of \eqref{pro:3}, we have} $\tp = \tpc$ and $\tq = \tqc$. {\red We now establish the relation of dual solutions $\tzeta$ and $\tzetac$ by studying KKT conditions of Problem \eqref{pro:3}}. To simplify notation, we denote the $k$-th component of the objective by
\begin{align*}
O_k(\bp_k, \bq_k) &= \biggl(\begin{smallmatrix}
\bp_k\\
\bq_k\\
\bl_k
\end{smallmatrix}\biggr)^T\biggl(\begin{smallmatrix}
Q_k &S_k^T & D_{k1}^T\\
S_k & R_k & D_{k2}^T\\
D_{k1} & D_{k2} & \0
\end{smallmatrix}\biggr)\biggl(\begin{smallmatrix}
\bp_k\\
\bq_k\\
\bl_k
\end{smallmatrix}\biggr) , \forall k\in[N-1],\\
O_N(\bp_N) &= \biggl(\begin{smallmatrix}
\bp_N\\
\bl_N
\end{smallmatrix}\biggr)^T\biggl(\begin{smallmatrix}
Q_N & D_N^T\\
D_N & \0
\end{smallmatrix}\biggr)\biggl(\begin{smallmatrix}
\bp_N\\
\bl_N
\end{smallmatrix}\biggr).
\end{align*}
Similarly, we define $\tilde{O}_k(\bp_k, \bq_k)$ and $\tilde{O}_N(\bp_N)$ by replacing $H_k, D_k$ by $\tilde{H}_k, \tilde{D}_k$. The KKT system of $\mL\mQ\mP$ is then given~by
\begin{equation}\label{eq:KKTLQP}
\hskip-9pt \0 = \left\{\begin{alignedat}{3}
& \nabla_{\bp_k}O_k(\bp_k, \bq_k) + \bzeta_{k-1} - A_k^T\bzeta_k, && \forall k\in[N-1],\\
& \nabla_{\bq_k}O_k(\bp_k, \bq_k)  - B_k^T\bzeta_k,  \hskip1.5cm  && \forall k\in[N-1],\\
&\nabla_{\bp_N}O_N(\bp_N) + \bzeta_{N-1}, \\
& \bp_{k+1} - (A_k\bp_k + B_k\bq_k  + C_k\bl_k), && \forall k\in[N-1],\\
& \bp_0 - \bl_{-1}.
\end{alignedat}\right.
\end{equation}
For the KKT system of $\mC\mL\mQ\mP$, we replace $\nabla O_k$ by $\nabla \tilde{O}_k$~and $\bzeta$ by $\bzeta^c$ in \eqref{eq:KKTLQP} {\red since two problems have the same linear-quadratic form.} By Algorithm \ref{alg:convex:proce}, we know that $\forall k\in[N-1]$,
\begin{align}\label{pequ:1}
\nabla_{\bp_k}\tilde{O}_k&(\bp_k, \bq_k) =  2\tilde{Q}_k\bp_k + 2\tilde{S}_k^T\bq_k + 2\tilde{D}_{k1}^T\bl_k \nonumber\\
= & 2(\hat{Q}_k - \bar{Q}_k)\bp_k + 2\tilde{S}_k^T\bq_k + 2\tilde{D}_{k1}^T\bl_k \nonumber\\
= & 2Q_k\bp_k + 2S_k^T\bq_k + 2D_{k1}^T\bl_k - 2\bar{Q}_k\bp_k \nonumber\\
& + 2A_k^T\bar{Q}_{k+1}(A_k\bp_k + B_k\bq_k + C_k\bl_k) \nonumber\\
= & \nabla_{\bp_k}O_k(\bp_k, \bq_k)- 2\bar{Q}_k\bp_k + 2A_k^T\bar{Q}_{k+1}\bp_{k+1},
\end{align}
where the last equality results from definition of $O_k$ and the $k$-th dynamic constraint. We can also show that
\begin{equation}\label{pequ:2}
\begin{aligned}
\nabla_{\bq_k}\tilde{O}_k(\bp_k, \bq_k) =& \nabla_{\bq_k}O_k(\bp_k, \bq_k) + 2B_k^T\bar{Q}_{k+1}\bp_{k+1},\\
\nabla_{\bp_N}\tilde{O}_N(\bp_N) = & \nabla_{\bp_N}O_N(\bp_N)  - 2\bar{Q}_N\bp_N. 
\end{aligned}
\end{equation}
Plugging \eqref{pequ:1}, \eqref{pequ:2} back into \eqref{eq:KKTLQP}, we obtain that the KKT system of $\mL\mQ\mP$ is {\em equivalent} to
\begin{equation*}
\0 = \left\{\begin{alignedat}{3}
&\nabla_{\bp_k}\tilde{O}_k + (\bzeta_{k-1} + 2\bar{Q}_k\bp_k) - A_k^T(\bzeta_k + 2\bar{Q}_{k+1}\bp_{k+1}),\\
& \nabla_{\bq_k}\tilde{O}_k - B_k^T(\bzeta_k + 2\bar{Q}_{k+1}\bp_{k+1}),  \hskip0.8cm  \forall k\in[N-1],\\
&\nabla_{\bp_N}\tilde{O}_N + (\bzeta_{N-1} + 2\bar{Q}_N\bp_N), \\
& \bp_{k+1} - (A_k\bp_k + B_k\bq_k  + C_k\bl_k), \hskip1cm \forall k\in[N-1],\\
& \bp_0 - \bl_{-1}.
\end{alignedat}\right.
\end{equation*}
Comparing the above equation with the KKT system of $\mC\mL\mQ\mP$, and using the invariance of the primal solution, we~see that $(\tpc, \tqc, \tzeta + 2\bar{Q}\tp)$ satisfies the KKT system of $\mC\mL\mQ\mP$. Since LICQ holds for $\mC\mL\mQ\mP$, the dual solution is unique. This implies $\tzetac = \tzeta + 2\bar{Q}\tp$ and we complete the proof.

\subsection{Proof of Theorem \ref{thm:2}}\label{pf:thm:2}

First of all, the invertibility of $W_k$ is guaranteed by~Assumption \ref{ass:1}, as directly shown in \cite[Lemma 3.5(i)]{Na2020Exponential}. We use reverse induction to prove the formula of $\tzeta_k$. According to \eqref{eq:KKTLQP}, for $k = N-1$ we have
\begin{equation*}
\tzeta_{N-1} = -\nabla_{\bp_N}O_N(\tp_N) = -2Q_N\tp_N - 2D_N^T\bl_N,
\end{equation*}
which satisfies \eqref{equ:dual:form} and proves the first induction step. Suppose $\tzeta_k$ satisfies \eqref{equ:dual:form}. From \eqref{eq:KKTLQP}, we have 	
\begin{align*}
\tzeta_{k-1} = &  A_k^T\tzeta_k - \nabla_{\bp_k}O_k(\tp_k, \tq_k)\\
= & A_k^T\tzeta_k - 2Q_k\tp_k - 2S_k^T\tq_k - 2D_{k1}^T\bl_k.
\end{align*}	
Plugging the expression for $\tzeta_k$ from \eqref{equ:dual:form}, we get
\begin{align*}
&\tzeta_{k-1} = -2A_k^TK_{k+1}\tp_{k+1} + 2A_k^T\big(\sum_{i=k+1}^N(M_i^{k+1})^T\bl_i \\
& \quad + \sum_{i=k+1}^{N-1}(V_i^{k+1})^TC_i\bl_i\big) - 2Q_k\tp_k - 2S_k^T\tq_k - 2D_{k1}^T\bl_k\\
&=  -2(A_k^TK_{k+1}A_k + Q_k)\tp_k - 2(S_k + B_k^TK_{k+1}A_k)^T\tq_k\\
& \quad + 2A_k^T\cbr{\sum_{i=k+1}^N(M_i^{k+1})^T\bl_i  + \sum_{i=k+1}^{N-1}(V_i^{k+1})^TC_i\bl_i}\\
& \quad -2A_k^TK_{k+1}C_k\bl_k - 2D_{k1}^T\bl_k\\
& = -2(A_k^TK_{k+1}A_k + Q_k)\tp_k + 2P_k^TW_k\tq_k \\
& \quad + 2A_k^T\cbr{\sum_{i=k+1}^N(M_i^{k+1})^T\bl_i  + \sum_{i=k+1}^{N-1}(V_i^{k+1})^TC_i\bl_i}\\
& \quad -2A_k^TK_{k+1}C_k\bl_k - 2D_{k1}^T\bl_k,
\end{align*}	 
where the second equality follows from $\tp_{k+1}-(A_k\tp_k+B_k\tq_k+C_k\bl_k)=\0$, and the third equality follows from the definition of $P_k$. By \cite[Lemma 3.5(ii)]{Na2020Exponential}, we have	
\begin{align*}
\tq_k = &P_k\tp_k + W_k^{-1}B_k^T\sum_{i=k+1}^{N}(M_i^{k+1})^T\bl_i - W_k^{-1}D_{k2}^T\bl_k\\
& + W_k^{-1}B_k^T\sum_{i=k+1}^{N-1}(V_i^{k+1})^TC_i\bl_i - W_k^{-1}B_k^TK_{k+1}C_k\bl_k.
\end{align*}
Combining the above two displays, we obtain	
\begin{align*}
&\tzeta_{k-1} 
= -2(A_k^TK_{k+1}A_k + Q_k)\tp_k + 2P_k^TW_k\bigg(P_k\tp_k \\
& \quad + W_k^{-1}B_k^T\cbr{\sum_{i=k+1}^N(M_i^{k+1})^T\bl_i + \sum_{i=k+1}^N(V_i^{k+1})^TC_i\bl_i}\\
& \quad - W_k^{-1}B_k^TK_{k+1}C_k\bl_k - W_k^{-1}D_{k2}^T\bl_k\bigg) - 2D_{k1}^T\bl_k\\
& \quad + 2A_k^T\bigg(\sum_{i=k+1}^N(M_i^{k+1})^T\bl_i  + \sum_{i=k+1}^{N-1}(V_i^{k+1})^TC_i\bl_i\bigg)\\
& \quad - 2A_k^TK_{k+1}C_k\bl_k\\
& =  -2(A_k^TK_{k+1}A_k + Q_k - P_k^TW_kP_k)\tp_k \\
& \quad + 2E_k^T\bigg(\sum_{i=k+1}^N(M_i^{k+1})^T\bl_i  + \sum_{i=k+1}^{N-1}(V_i^{k+1})^TC_i\bl_i\bigg)\\
& \quad -2E_k^TK_{k+1}C_k\bl_k - 2(D_{k1} + D_{k2}P_k)^T\bl_k\\
& = -2K_k\tp_k + 2\sum_{i=k}^N(M_i^k)^T\bl_i + 2\sum_{i=k}^N(V_i^k)^TC_i\bl_i,
\end{align*}
where the second equality follows from the definition of $E_k$ and the third equality follows from definitions of $K_k$, $M^k_i$, and $V^k_i$. This verifies the induction step and finishes the proof.

\subsection{Proof of Lemma \ref{lem:dual}}\label{pf:lem:dual}
	
We use the closed form of $\tzetac$ established in Theorem \ref{thm:2}. We mention that all matrices are calculated based on $\{\tilde{H}_k, \tilde{D}_k\}$. {\red For any $k\in[-1, N-1]$, we consider two cases.}

\noindent{\bf (a)} $\bl = \be_i$ for $i\in[N]$. {\red We then have three subcases.}

\noindent{\bf (a1)} $i\leq k$. In this case, we apply \eqref{equ:dual:form} and immediately have for some constant $\Upsilon_1>0$ that
\begin{equation}\label{equ:b}
\|\tzetac_k\| = \|-2K_{k+1}\tp_{k+1}\|\leq 2\|K_{k+1}\|\|\tp_{k+1}\|\leq {\red 2\Upsilon_1\Upsilon}\rho^{k+1-i}.
\end{equation}
Here, the last inequality is due to Theorem \ref{thm:primal:sensitivity} that $\|\bp_{k+1}^\star\| \leq \Upsilon\rho^{k+1-i}$, and the fact that $\|K_{k+1}\|\leq \Upsilon_1$ for some constant $\Upsilon_1>0$, stated precisely in \cite[(4.7)]{Na2020Exponential}.

\noindent{\bf (a2)} $k+1\leq i\leq N-1$. We apply \eqref{equ:dual:form} and have for~some~constants $\Upsilon_2, \Upsilon_3 >0$ that
\begin{align*}
\|\tzetac_k\| = &\|-2K_{k+1}\tp_{k+1} + 2(M_i^{k+1})^T\be_i + 2(V_i^{k+1})^TC_i\be_i\|\\
\leq & {\red 2\Upsilon_1\Upsilon}\rho^{i-k-1} + {\red 2\Upsilon_2}\rho^{i-k-1} + {\red 2\Upsilon_3\Upsilon_{\text{upper}}}\rho^{i-k}\\
\leq & 2(\Upsilon_1\Upsilon + \Upsilon_2 + \Upsilon_3\Upsilon_{\text{upper}})\rho^{i-k-1}.
\end{align*}
The second inequality is due to \eqref{equ:b}, the fact that $\|C_i\|\leq \Upsilon_{\text{upper}}$ in Assumption \ref{ass:3}, and the fact that
\begin{equation*}
\|M_i^{k+1}\|\leq \Upsilon_2\rho^{i-k-1},\quad\quad \|V_i^{k+1}\|\leq \Upsilon_3\rho^{i-k}
\end{equation*}	
for some constants $\Upsilon_2, \Upsilon_3>0$, stated precisely in \cite[(5.11)]{Na2020Exponential}.  
	
\noindent{\bf (a3)} $i = N$. {\red Analogous to the derivations in {\bf (a2)}, we apply~\eqref{equ:dual:form} and have 
\begin{align*}
& \|\tzetac_k\| = \|-2K_{k+1}\tp_{k+1} + 2(M_N^{k+1})^T\be_N\|\\
& \leq  2\Upsilon_1\Upsilon\rho^{N-k-1} + 2\Upsilon_2\rho^{N-k-1}=  2(\Upsilon_1\Upsilon + \Upsilon_2 )\rho^{N-k-1}.
\end{align*}}
\hskip-3pt\noindent{\bf (b)} $\bl = \be_{-1}$. Applying \eqref{equ:dual:form} and Theorem \ref{thm:primal:sensitivity}, we obtain that
\begin{equation*}
\|\tzetac_k\| = \|-2K_{k+1}\tp_{k+1}\|\leq 2\Upsilon_1\Upsilon\rho^{k+1}.
\end{equation*}
Combining the above two cases, we let {\red $\Upsilon' = 2(\Upsilon_1\Upsilon + \Upsilon_2 + \Upsilon_3\Upsilon_{\text{upper}})$} and complete the proof.

\subsection{Proof of Theorem \ref{thm:dual:sensitivity}}\label{pf:thm:dual:sensitivity}
By Lemma \ref{lem:dual} and Theorem \ref{thm:1} we have for all $k\in[-1, N-1]$ that
\begin{equation*}
\|\tzeta_k\| = \|\tzetac_k - 2\bar{Q}_{k+1}\tp_{k+1}\|\leq  \|\tzetac_k\| + 2\|\bar{Q}_{k+1}\|\|\tp_{k+1}\|.
\end{equation*}
By \cite[Theorem 3.8, Claim 1]{Na2020Exponential} and \cite[Lemma 4.3]{Na2020Exponential}, we know $\bar{Q}_{k+1}$ (which is $\bar{Q}_{k+1}(\delta)$ in their context) satisfies $\|\bar{Q}_{k+1}\|\leq \Upsilon_Q$ for some constant $\Upsilon_Q$. Thus, by Theorem \ref{thm:primal:sensitivity}, we can let $\Upsilon'' =\Upsilon' + 2\Upsilon_Q\Upsilon$ to complete the proof.

\ifCLASSOPTIONcaptionsoff
  \newpage
\fi

\bibliographystyle{IEEEtran}
\bibliography{ref}

\begin{IEEEbiographynophoto}{Sen Na}
is a fifth-year Ph.D. student in the Department of Statistics at the University of Chicago under the supervision of Mihai Anitescu and Mladen Kolar. Before coming to UChicago, he received B.S. degree in mathematics from Nanjing University, China. His research interests lie in nonlinear dynamic programming, high-dimensional statistics, semiparametric modeling, and their interface. He is also serving as a reviewer of the SIAM Journal on Optimization, and Journal of Machine Learning Research.
\end{IEEEbiographynophoto}
\vspace{-0.3in}
\begin{IEEEbiographynophoto}{Sungho Shin}
is a Ph.D. candidate in the Department of Chemical and Biological Engineering at the University of Wisconsin-Madison. He received his B.S. in chemical engineering and mathematics from Seoul National University, South Korea, in 2016. His research interests include control theory and optimization algorithms for complex networks.
\end{IEEEbiographynophoto}
\vspace{-0.3in}
\begin{IEEEbiographynophoto}{Mihai Anitescu}
is a senior computational mathematician in the Mathematics and Computer Science Division at Argonne National Laboratory and a professor in the Department of Statistics at the University of Chicago. He obtained his engineer diploma (electrical engineering) from the Polytechnic University of Bucharest in 1992 and his Ph.D. in applied mathematical and computational sciences from the University of Iowa in 1997. He specializes in the areas of numerical optimization, computational science, numerical analysis, and uncertainty quantification. He is on the editorial board of the SIAM Journal on Optimization, and he is a senior editor for Optimization Methods and Software. He is a past member of the editorial boards of Mathematical Programming A and B,  SIAM Journal on Scientiﬁc Computing, and SIAM/ASA Journal in Uncertainty Quantification.
\end{IEEEbiographynophoto}
\vspace{-0.3in}
\begin{IEEEbiographynophoto}{Victor M. Zavala}
is the Baldovin-DaPra Associate Professor in the Department of Chemical and Biological Engineering at the University of Wisconsin-Madison. He holds a B.Sc. degree from Universidad Iberoamericana and a Ph.D. degree from Carnegie Mellon University, both in chemical engineering. He is an associate editor for the Journal of Process Control and for IEEE Transactions on Control and Systems Technology.  He is also a technical editor of Mathematical Programming Computation. His research interests are in the areas of energy systems, high-performance computing, stochastic programming, and predictive control.
\end{IEEEbiographynophoto}

\begin{flushright}
	\scriptsize \framebox{\parbox{3.4in}{Government License: The
			submitted manuscript has been created by UChicago Argonne,
			LLC, Operator of Argonne National Laboratory (``Argonne").
			Argonne, a U.S. Department of Energy Office of Science
			laboratory, is operated under Contract
			No. DE-AC02-06CH11357.  The U.S. Government retains for
			itself, and others acting on its behalf, a paid-up
			nonexclusive, irrevocable worldwide license in said
			article to reproduce, prepare derivative works, distribute
			copies to the public, and perform publicly and display
			publicly, by or on behalf of the Government. The Department of Energy will provide public access to these results of federally sponsored research in accordance with the DOE Public Access Plan. http://energy.gov/downloads/doe-public-access-plan. }}
	\normalsize
\end{flushright}

\end{document}